\documentclass[review]{elsarticle}
\usepackage{srcltx}
\usepackage{eurosym}
\usepackage{mathtools}
\usepackage{amsmath}
\usepackage{amsfonts}
\usepackage{amssymb}
\usepackage{amsthm}
\usepackage{graphicx}
\usepackage{mathrsfs}
\usepackage{xcolor}
\usepackage{exscale}
\usepackage{latexsym}
\usepackage[colorlinks,plainpages=true,pdfpagelabels,hypertexnames=true,colorlinks=true,pdfstartview=FitV,linkcolor=blue,citecolor=red,urlcolor=black]{hyperref}
\PassOptionsToPackage{unicode}{hyperref}	
\PassOptionsToPackage{naturalnames}{hyperref}
\usepackage{enumerate}
\usepackage[shortlabels]{enumitem}
\usepackage{bookmark}
\usepackage{wasysym}
\usepackage{esint}
\usepackage[ddmmyyyy]{datetime}
\usepackage[margin=2.2cm]{geometry}

\makeatletter
\g@addto@macro\normalsize{%
  \setlength\abovedisplayskip{4pt}
  \setlength\belowdisplayskip{1pt}
  \setlength\abovedisplayshortskip{4pt}
  \setlength\belowdisplayshortskip{1pt}
}
\def\arraystretch{1.2}
\numberwithin{equation}{section}
\everymath{\displaystyle}
\newtheorem{theorem}{Theorem}[section]
\newtheorem{lemma}[theorem]{Lemma}
\newtheorem{corollary}[theorem]{Corollary}
\newtheorem{proposition}[theorem]{Proposition}

\newtheorem{definition}[theorem]{Definition}
\newtheorem{remark}[theorem]{Remark}        
  
\numberwithin{equation}{section}
\def\aa{\mathcal{A}}

\newcommand{\aaa}{\overline{\mathcal{A}}}
\newcommand{\lamot}{\La_0,\La_1}

\newcommand{\pa}{\partial}
\newcommand{\ddt}[1]{\frac{d#1}{dt}}
\newcommand{\dds}[1]{\frac{d#1}{ds}}

\newcommand{\vlh}{\lsbt{v}{\la,h}}
\newcommand{\vl}{\lsbt{v}{\la}}
\newcommand{\elam}{\lsbo{E}{\lambda}}
\newcommand{\bff}{{\bf f}}

\newcommand{\mm}{\mathcal{M}}

\newcommand{\tQ}{\tilde{Q}}

\newcommand{\tom}{\tilde{\Om}}

\newcommand{\tz}{\tilde{z}}

\makeatletter
\newcommand{\vo}{\vec{o}\@ifnextchar{^}{\,}{}}
\makeatother
\def\Yint#1{\mathchoice
    {\YYint\displaystyle\textstyle{#1}}%
    {\YYint\textstyle\scriptstyle{#1}}%
    {\YYint\scriptstyle\scriptscriptstyle{#1}}%
    {\YYint\scriptscriptstyle\scriptscriptstyle{#1}}%
      \!\iint}
\def\YYint#1#2#3{{\setbox0=\hbox{$#1{#2#3}{\iint}$}
    \vcenter{\hbox{$#2#3$}}\kern-.50\wd0}}
\def\longdash{-\mkern-9.5mu-} 
\def\tiltlongdash{\rotatebox[origin=c]{18}{$\longdash$}}
\def\fiint{\Yint\tiltlongdash}
\def\Xint#1{\mathchoice
    {\XXint\displaystyle\textstyle{#1}}%
    {\XXint\textstyle\scriptstyle{#1}}%
    {\XXint\scriptstyle\scriptscriptstyle{#1}}%
    {\XXint\scriptscriptstyle\scriptscriptstyle{#1}}%
      \!\int}
\def\XXint#1#2#3{{\setbox0=\hbox{$#1{#2#3}{\int}$}
    \vcenter{\hbox{$#2#3$}}\kern-.50\wd0}}
\def\hlongdash{-\mkern-13.5mu-}
\def\tilthlongdash{\rotatebox[origin=c]{18}{$\hlongdash$}}
\def\hint{\Xint\tilthlongdash}
\makeatletter
\def\namedlabel#1#2{\begingroup
   \def\@currentlabel{#2}%
   \label{#1}\endgroup
}
\makeatother
\makeatletter
\newcommand{\rmh}[1]{\mathpalette{\raisem@th{#1}}}
\newcommand{\raisem@th}[3]{\hspace*{-1pt}\raisebox{#1}{$#2#3$}}
\makeatother

\newcommand{\lsb}[2]{#1_{\rmh{-3pt}{#2}}}
\newcommand{\lsbo}[2]{#1_{\rmh{-1pt}{#2}}}
\newcommand{\lsbt}[2]{#1_{\rmh{-2pt}{#2}}}
\newcommand{\redref}[2]{\texorpdfstring{\protect\hyperlink{#1}{\textcolor{black}{(}\textcolor{red}{#2}\textcolor{black}{)}}}{}}
\newcommand{\redlabel}[2]{\hypertarget{#1}{\textcolor{black}{(}\textcolor{red}{#2}\textcolor{black}{)}}}

\newcommand{\descitem}[2]{\item[#1] \label{#2}}
\newcommand{\descref}[2]{\hyperref[#1]{\textnormal{\textcolor{black}{(}\textcolor{blue}{\bf #2}\textcolor{black}{)}}}}

\newcommand{\dref}[2]{\hyperref[#1]{\textcolor{black}{(}\textcolor{blue}{\bf #2}\textcolor{black}{)}}}

\newcommand{\tk}{\tilde{k}}

\newcommand{\mfx}{\mathfrak{x}}

\newcommand{\mfz}{\mathfrak{z}}

\newcommand{\mft}{\mathfrak{t}}

\newcommand{\mch}{\mathcal{H}}

\newcommand\RR{\mathbb{R}}

\newcommand\ZZ{\mathbb{Z}}
\newcommand\NN{\mathbb{N}}
\newcommand{\al}{\alpha}
\newcommand{\be}{\beta}
\newcommand{\ga}{\gamma}
\newcommand{\de}{\delta}
\newcommand{\ve}{\varepsilon}

\newcommand{\ep}{\epsilon}
\newcommand{\ka}{\kappa}

\newcommand{\om}{\omega}
\newcommand{\la}{\lambda}
\newcommand{\vt}{\vartheta}

\newcommand{\Om}{\Omega}

\newcommand{\La}{\Lambda}
\newcommand{\tTh}{{\Upsilon}}
\DeclareMathOperator{\dv}{div}
\DeclareMathOperator{\spt}{spt}

\DeclareMathOperator{\diam}{diam}

\newcommand{\iprod}[2]{\langle #1 \ ,  #2\rangle}
\newcommand{\norm}[1]{\left|\hspace{-0.2mm}\left| #1 \right|\hspace{-0.2mm}\right|}
\newcommand{\abs}[1]{\left| #1\right|}

\newcommand{\mgh}[1]{\left\{ #1\right\}}

\newcommand{\lbr}[1][(]{\left#1}
\newcommand{\rbr}[1][)]{\right#1}

\newcommand{\avgs}[2]{\lsbo{\lbr #1 \rbr}{#2}}

\newcommand{\txt}[1]{\qquad \text{#1} \qquad}

\newcommand{\omt}{ \Om \times (-T,T)}

\newcommand{\ov}{\overline{V}}

\newcommand{\zv}{\zeta_{\ve}}
\newcommand{\gflt}{$(\ga,S_0)$-Reifenberg flat }
\newcommand{\ors}{Q_{\rho,s}(\mfz)}
\newcounter{whitney}
\refstepcounter{whitney}

\newcounter{ineqcounter}
\refstepcounter{ineqcounter}
\usepackage[titletoc,toc,page]{appendix}
\makeatletter
\def\ps@pprintTitle{%
 \let\@oddhead\@empty
 \let\@evenhead\@empty
 \def\@oddfoot{}%
 \let\@evenfoot\@oddfoot}
\makeatother

\begin{document}

\begin{frontmatter}

\title{Gradient weighted estimates at the natural exponent for Quasilinear Parabolic equations.}

\author[myaddress]{Karthik Adimurthi\corref{mycorrespondingauthor}\tnoteref{thanksfirstauthor}}
\cortext[mycorrespondingauthor]{Corresponding author}
\ead{karthikaditi@gmail.com and kadimurthi@snu.ac.kr}
\tnotetext[thanksfirstauthor]{Supported by the National Research Foundation of Korea grant NRF-2017R1A2B2003877.}

\author[myaddress,myaddresstwo]{Sun-Sig Byun\tnoteref{thankssecondauthor}}
\ead{byun@snu.ac.kr}
\tnotetext[thankssecondauthor]{Supported by the National Research Foundation of Korea grant  NRF-2015R1A4A1041675. }

\address[myaddress]{Department of Mathematical Sciences, Seoul National University, Seoul 08826, Korea.}
\address[myaddresstwo]{Research Institute of Mathematics, Seoul National University, Seoul 08826, Korea.}

\begin{abstract}
In this paper, we obtain weighted norm inequalities for the spatial gradients of weak solutions to quasilinear parabolic equations with weights in the Muckenhoupt class $A_{\frac{q}{p}}(\RR^{n+1})$ for $q\geq p$ on non-smooth domains. Here the quasilinear nonlinearity is  modelled after the standard $p$-Laplacian operator. Until now, all the weighted estimates for the gradient were obtained only for exponents $q>p$. The results for exponents $q>p$ used the full complicated machinery of the Calder\'on-Zygmund theory developed over the past few decades, but the constants blow up as $q \rightarrow p$ (essentially because the Maximal function is not bounded on $L^1$). 

In order to prove the weighted estimates for the gradient at the natural exponent, i.e., $q=p$, we need to obtain improved a priori estimates below the natural exponent. To this end, we develop the technique of Lipschitz truncation based on \cite{AdiByun2,KL} and obtain significantly improved  estimates below the natural exponent. Along the way, we also obtain improved, unweighted Calder\'on-Zygmund type estimates below the natural exponent which is new even for the linear equations.
\end{abstract}


\end{frontmatter}

\tableofcontents

\section{Introduction}
\label{section_zero}
In this paper, we are interested in obtaining Calder\'on-Zygmund type regularity estimates in weighted Lebesgue spaces for equations of the form
\begin{equation}
\label{main}
 \left\{
 \begin{array}{ll}
u_t - \dv  \aa(x,t,\nabla u) = \dv |\bff|^{p-2} \bff  & \text{in} \ \Om \times (-T,T), \\
u = 0 & \text{on} \ \partial_p\lbr  \Om \times (-T,T)\rbr,
 \end{array}
\right.
\end{equation}
where the nonlinearity $ \aa(x,t,\nabla u)$ is modelled after the well studied $p$-Laplacian operator given by $|\nabla u|^{p-2} \nabla u$ in a bounded domain $\Om\subset \RR^n$ with $n\geq 2$,  potentially with non-smooth boundary $\pa \Om$. The parabolic boundary is given by $\pa_p\lbr  \Om \times (-T,T)\rbr := \pa \Om \times (-T,T) \bigcup \Om \times \{-T\}$.

Over the past decades, there have been a plethora of a priori estimates of the Calder\'on-Zygmund type obtained for \eqref{main}. We shall point out that all the estimates discussed in the introduction are quantitative, but in order to highlight the novelty of the results in this paper, we shall only mention the qualitative nature of the estimates existing in the literature. 

The first extension of the Calder\'on-Zygmund theory for \eqref{main} with  $\aa(x,t,\nabla u) = |\nabla u|^{p-2} \nabla u$ for $p > \frac{2n}{n+2}$ (note that this restriction is natural for parabolic problems, see \cite[Chapter 5]{DiB1}) was obtained in \cite{AM2},  where they proved 
\[
|\bff| \in L^q_{loc} \quad \Longrightarrow\quad |\nabla u| \in L^q_{loc} \txt{for all} q \geq p.
\]

Since then, many extensions were obtained which generalized the estimates in \cite{AM2} to more general nonlinearities, function spaces and up to the boundary (see \cite{MR3090083,MR3143810,MR3461425,BOS1,MR2866816,MR2836359} and the references therein). \emph{In this paper, the first result we will prove is  an improved global a priori estimate of the form}
\[
|\bff| \in L^q (\Om_T) \quad \Longrightarrow\quad |\nabla u| \in L^q (\Om_T) \txt{for all} q \in [p-\be_0,p],
\]
where $\be_0$ is a sufficiently small universal exponent. The improvement is two fold, firstly this estimate is obtained below the natural exponent and secondly, the estimate assumes no regularity of the coefficients and hence is non-perturbative. \emph{As a consequence, this result is new even for linear equations.}

\emph{The second result that we are interested in obtaining is global estimates in weighted Lebesgue spaces with the weight in Muckenhoupt class.}
%
%
For general nonlinear structures with linear growth, i.e., $\aa(x,t,\nabla u) \approx \nabla u$  with the coefficients satisfying a small bounded mean oscillation restriction, the following global weighted estimates was obtained in \cite{byun2013weighted}:
\[
|\bff| \in L^q_{\om}(\Om_T) \quad \Longrightarrow\quad |\nabla u| \in L^q_{\om}(\Om_T) \txt{for all} q > 2 \txt{and} \om \in A_{\frac{q}{2}}(\RR^{n+1}).
\]
Note that in particular, they cannot consider $q=2$ in \cite{byun2013weighted}.

Subsequently, in \cite{byun2017weighted}, they were able to prove analogous results for nonlinearities of the form $\aa(x,t,\nabla u) \approx  |\nabla u|^{p-2} \nabla u$ with $\frac{2n}{n+2} < p <\infty$ and more general Weighted Orlicz spaces, in particular, they prove
\[
|\bff| \in L^q_{\om}(\Om_T) \quad \Longrightarrow\quad |\nabla u| \in L^q_{\om}(\Om_T) \txt{for all} q > p \txt{and} \om \in A_{\frac{q}{p}}(\RR^{n+1}).
\]
Note that in particular, they cannot consider $q=p$ in \cite{byun2017weighted}.

The main obstacle in proving weighted estimates at $q=p$ is due to the failure of strong $L^1-L^1$ bounds for the Hardy-Littlewood Maximal function. Therefore to reach the natural exponent, a different approach is needed. In this paper, we achieve this result by showing the weighted estimate holds with $q=p$, i.e., \eqref{intro_est} holds. To overcome this difficulty, we construct a suitable test function based on a modification of  the techniques developed in \cite{AdiByun2,KL} and obtain estimates below the natural exponent, i.e., under suitable restrictions on the $\aa(x,t,\nabla u)$ and $\Om$ (similar to those in \cite{byun2017weighted}), we prove
\begin{equation}\label{intro_est}
|\bff| \in L^q_{\om}(\Om_T) \quad \Longrightarrow\quad |\nabla u| \in L^q_{\om}(\Om_T) \txt{for all} q \geq p \txt{and} \om \in A_{\frac{q}{p}}(\RR^{n+1}).
\end{equation}

There are a few remarks to be made;  firstly the estimate \eqref{intro_est} represents an end point weighted estimate for quasilinear parabolic equations; secondly, the optimal weight class in the elliptic case is conjectured to be $A_{\frac{q}{p-1}}$ (see \cite[Theorem 1.9]{AP2} for more on this and the elliptic Iwaniec conjecture) and in the parabolic case too, the optimal result is expected to be of the form
\[
|\bff| \in L^q_{\om}(\Om_T) \quad \Longrightarrow\quad |\nabla u| \in L^q_{\om}(\Om_T) \txt{for all} q > p-1 \txt{and} \om \in A_{\frac{q}{p-1}}(\RR^{n+1}),
\]
but this result seems to be far out of reach of current methods.

The plan of the paper is as follows: In Section \ref{section_one}, we collect all the assumptions on the domain, nonlinear structure and the weight class along with some preliminary well known results, in Section \ref{section_two}, we will describe the main theorem that will be proved, in Section \ref{section_three}, we will develop a general Lipschitz truncation technique and construct a suitable test function, in Section \ref{section_four}, we will define useful perturbations of \eqref{main} and prove  crucial difference estimates below the natural exponent, in Section \ref{section_five}, we will prove Theorem \ref{thm6.1}, in Section \ref{section_six}, we will use standard covering arguments to prove the parabolic analogue of a good-$\la$ estimate and finally use that in Section \ref{section_seven} to prove Theorem \ref{main_theorem}.

\subsection*{Acknowledgement} 
The authors would like to thank the organisers of the conference \emph{Recent developments in Nonlinear Partial Differential Equations and Applications - NPDE2017} held at TIFR-CAM, Bangalore where part of this work was done.

\section{Preliminaries}
\label{section_one}

The following restriction on the exponent $p$ will always be enforced:
\begin{equation}
 \label{restriction_p}
 \frac{2n}{n+2} < p < \infty.
\end{equation}

\begin{remark}The restriction in \eqref{restriction_p} is necessary when dealing with parabolic problems because,  we invariably have to deal with the $L^2$-norm of the solution which comes from the time-derivative. On the other hand, the following Sobolev  embedding $W^{1,p} \hookrightarrow L^2$ is true provided \eqref{restriction_p} holds.  On the other hand, if we assume $u \in L^{r}(\Om_T)$ for some $r\geq 1$ such that $\La_r:=n(p-2) + rp >0$ (see \cite[Chapter 5]{DiB1} for more on this), then we can obtain analogous result as to Theorem \ref{main_theorem}. This extension of Theorem \ref{main_theorem} to the case $1< p \leq \frac{2n}{n+2}$ requires only a technical modification provided $\La_r > 0$ and will be omitted. \end{remark}

\subsection{Assumptions on the Nonlinear structure}
\label{nonlinear_structure}
We shall now collect the assumptions on the nonlinear structure in \eqref{main}. 
We assume that $\aa(x,t,\nabla u)$ is a Carath\'eodory function, i.e., we have $(x,t) \mapsto \aa(x,t,\zeta)$  is measurable for every $\zeta \in \RR^n$ and 
$\zeta \mapsto \aa(x,t,\zeta)$ is continuous for almost every  $(x,t) \in \Om \times (-T,T)$. We also assume $\aa(x,t,0) = 0$ and $\aa(x,t,\zeta)$ is differentiable in $\zeta$ away from the origin, i.e., $d_{\zeta}\aa(x,t,\zeta)$ exists for a.e. $(x,t) \in \RR^{n+1}$.

We further assume that for a.e. $(x,t) \in \Om \times (-T,T)$ and for any $\eta,\zeta \in \RR^n$, there exists two given positive constants $\lamot$ such that  the following bounds are satisfied   by the nonlinear structures :
\begin{gather}
\iprod{\aa(x,t,\zeta - \aa(x,t,\eta)}{\zeta - \eta}\geq  \La_0 \lbr |\zeta|^2 + |\eta|^2 \rbr^{\frac{p-2}{2}} |\zeta - \eta|^2, \label{abounded}\\
|\aa(x,t,\zeta) - \aa(x,t,\eta)| \leq \La_1  |\zeta-\eta|\lbr |\zeta|^2 + |\eta|^2 \rbr^{\frac{p-2}{2}}. \label{bbounded}
\end{gather}

Note that from the assumption $\aa(x,t,0) = 0$, we get for a.e. $(x,t) \in \RR^{n+1}$, there holds
\begin{equation*}
\label{bound_b}
|\aa(x,t,\zeta)| \leq \La_1 |\zeta|^{p-1}.
\end{equation*}

\subsection{Structure of \texorpdfstring{$\Om$}.}
The domain that we consider may be non-smooth but should satisfy some regularity condition. This condition would essentially say that at each boundary point and every scale, we require the boundary of the domain to be between two hyperplanes separated by a distance  proportional to the scale.  

\begin{definition}
\label{reif_flat}
Given any $\ga \in (0,1]$ and $S_0 >0$, we say that $\Om$ is \gflt domain if for every $x_0 \in \pa \Om$ and every $r \in (0,S_0]$, there exists a system of coordinates $\{y_1,y_2,\ldots,y_n\}$ (possibly depending on $x_0$ and $r$) such that in this coordinate system, $x_0 =0$ and 
\[
B_r(0) \cap \{y_n > \ga r\} \subset B_r(0) \subset B_r(0) \cap \{y_n > -\ga r\}.
\]
\end{definition}

The class of Reifenberg flat domains is standard in obtaining Calder\'on-Zygmund type estimates, in the elliptic case, see \cite{AP2,BO,BO1,BW-CPAM} and the references therein whereas for the parabolic case, see \cite{Bui1,MR3461425,BOS1,MR2836359} and the references therein. 

From the definition of \gflt domains, it is easy to see that the following property holds:
\begin{lemma}
\label{measure_density}
Let $\ga >0$ and $S_0>0$ be given and suppose that $\Om$ is a \gflt domain, then there exists an $m_e = m_e(\ga,S_0,n)\in(0,1)$ such that for every $x \in {\Om}$ and every $r >0$, there holds
\begin{equation}\label{msd}
|\Om^c \cap B_r(x)| \geq m_e |B_r(x)|.
\end{equation}
\end{lemma}

\subsection{Smallness Assumption}
In order to prove the main results, we need to assume a smallness condition satisfied by $(\aa,\Om)$.
\begin{definition}
Let $\ga \in (0,1)$ and $S_0 >0$ be given, we then say $(\aa,\Om)$ is $(\ga,S_0)$-vanishing if the following assumptions hold:
\begin{description}[leftmargin=*]
  \item[(i) Assumption on $\aa$:] For any parabolic cylinder $Q_{\rho,s}(\mfz)$ centered at $\mfz := (\mfx,\mft)\in \RR^{n+1}$, let us define the following:
  \begin{equation*}
   \label{a_difference}
   \Theta(\aa,Q_{\rho,s}(\mfz))(x,t) := \sup_{\zeta \in \RR^n\setminus \{0\}} \frac{\abs{\aa(x,t,\zeta) - \aaa_{B_{\rho}(\mfx)}(t,\zeta)}}{|\zeta|^{p-1}},
  \end{equation*}
where we have used the notation 
\begin{equation}
\label{def_aaa}
\aaa_{B_{\rho}(\mfx)}(t,\zeta) := \fint_{B_{\rho}(\mfx)} \aa(x,t,\zeta) \ dx.
\end{equation}
Then $\aa$ is said to be $(\ga,S_0)$ vanishing if for some $\tau \in [1,\infty)$, there holds
\begin{equation}
 \label{small_aa}
 [\aa]_{\tau,S_0} := \sup_{\substack{0<\rho\leq S_0\\0<s\leq S_0^2}} \sup_{\mfz \in  \RR^{n+1}} \fint_{Q_{\rho,s}(\mfz)} \abs{\Theta(\aa,B_{\rho}(\mfx))(z)}^{\tau} \ dz \leq \ga^{\tau}.
\end{equation}
Here we have used the notation $z := (x,t) \subset \RR^{n+1}$.

\item[(ii) Assumption on $\pa \Om$:] We ask that $\Om$ is a \gflt  in the sense of  Definition \ref{reif_flat}. 
 \end{description}
\end{definition}
  
\begin{remark}
From \eqref{abounded}, we see that $|\Theta(\aa,Q_{\rho,s}(\mfz))(x,t)| \leq 2 \La_1$, thus combining this with the assumption \eqref{small_aa}, we see from standard interpolation inequality that for any $1 \leq \mathfrak{t} <\infty$, there holds
\[
\fint_{Q_{\rho,s}(\mfz)} |\Theta(\aa,Q_{\rho,s}(\mfz))(z)|^{\mathfrak{t}} \ dz \leq C(\ga,\La_1),
\]
with $C(\ga, \La_1) \rightarrow 0$ whenever $\ga \rightarrow 0$. 
\end{remark}

\subsection{Muckenhoupt weights}

In this subsection, let us collect all the properties of the weights that will be considered in the paper. See \cite[Chapter 9]{grafakos2004classical} for the details concerning this subsection.
\begin{definition}[Strong Muckenhoupt Weight]
\label{muck_weight}
A non negative, locally integrable function $\om$ is a \emph{strong weight} in $A_q(\RR^{n+1})$ for some $1 < q < \infty$ if 
\begin{equation*}
\label{m_wght}
\sup_{\mfz \in \RR^{n+1}}\sup_{\substack{0<\rho<\infty,\\0<s<\infty}} \lbr \fiint_{Q_{\rho,s}(\mfz)} \om(z) \ dz \rbr \lbr \fiint_{Q_{\rho,s}(\mfz)}\om^{\frac{-1}{q-1}}(z) \ dz \rbr^{q-1} =: [\om]_{q} < \infty.
\end{equation*}

In the case $q =1$, we define the \emph{strong $A_1(\RR^{n+1})$ weight} to be the class of non negative, locally integrable function $\om \in A_1(\RR^{n+1})$ satisfying
\begin{equation*}
\label{m_wght_1}
\sup_{z \in \RR^{n+1}}\sup_{\substack{0<\rho<\infty,\\0<s<\infty}} \lbr \fiint_{Q_{\rho,s}(\mfz)} \om(z) \ dz \rbr \norm{\om^{-1}}_{L^{\infty}(Q_{\rho,s}(\mfz))} =: [\om]_{1} < \infty.
\end{equation*}

The quantity $[w]_q$ for $1\leq q < \infty$ will be called as the $A_q$ constant of the weight $\om$. 
\end{definition}

We will need the following important characterization of Muckenhoupt weights:
\begin{lemma}
\label{weight_lemma}
A parabolic weight $w \in A_q$ for $1<q<\infty$ if and only if 
\[
\lbr \frac{1}{|Q|} \iint_Q f(x,t)  \ dx \ dt \rbr^q \leq \frac{c}{w(Q)} \iint_Q |f(x,t)|^q w(x,t) \ dx \ dt,
\]
holds for all non-negative, locally integrable functions $f$ and all cylinders $Q = Q_{\rho,s}(x,t)$.
\end{lemma}
%

As a direct consequence of Lemma \ref{weight_lemma}, the following  Lemma holds:

\begin{lemma}
Let $\om \in A_q(\RR^{n+1})$ for some $1<q<\infty$, then there exists positive constants $c = c(n,q,[\om]_q)$ and $\tau = \tau(n,q,[\om]_q) \in (0,1)$ such that 
\begin{equation*}
\label{doubling}
\frac{1}{c}\lbr \frac{|E|}{|Q|} \rbr^q \leq  \frac{\om(E)}{\om(Q)} \leq c \lbr \frac{|E|}{|Q|} \rbr^{\tau},
\end{equation*}
for all $E \subset Q$ and all parabolic cylinders $Q_{\rho,s}(\mfz)$. 

\end{lemma}

Another important result regarding the strong Muckenhoupt weights that will be needed  is the following self-improvement property:
\begin{lemma}
\label{reverse_holder}
Let $1<q<\infty$ and suppose $\om \in A_q$ be a given weight, then there exists an $\ve_0 = \ve_0(n,q,[\om]_q)>0$ such that $\om \in A_{q-\ve_0}$ with the estimate $[\om]_{q-\ve_0} \leq C [\om]_q$ where $C = C{(q,n,[\om]_q)}$. 
\end{lemma}
%

We will now define the $A_{\infty}$ class as follows:
\begin{definition}
\label{a_infinity}
A weight $\om \in A_{\infty}$ if and only if there are constants $\tau_0, \tau_1 >0$ such that for every  parabolic cylinder $Q=Q_{\rho,s} \subset \RR^{n+1}$ and every measurable $E \subset Q$, there holds
\[
\om(E) \leq \tau_0 \lbr \frac{|E|}{|Q|} \rbr^{\tau_1} \om(Q).
\]

Moreover, if $\om$ is an $A_q$ weight with $[\om]_q \leq \overline{\omega}$, then the constants $\tau_0$ and $\tau_1$ can be chosen such that $\max\{ \tau_0,1/\tau_1 \} \leq c(\overline{\om},n)$. 
\end{definition}

From the general theory of Muckenhoupt weights, we see that $A_{\infty} = \bigcup_{1 \leq q <\infty} A_q$. 

\begin{remark}
The weight class considered in Definition \ref{muck_weight} is called Strong Muckenhoupt class because the cylinders  are decoupled in space and time, i.e., $\rho$ and $s$ are not related when considering cylinders $Q_{\rho,s}$. When considering linear equations (i.e., $p=2$), the weight class is defined with respect to cylinders of the form $Q_{\rho,\rho^2}$.  This is possible because in the case $p=2$, there is an invariance property under normalization, which does not exist if $p \neq 2$. It is an open question if one can obtain the results of this paper for Muckenhoupt weights defined with respect to cylinders belonging to a more restricted class (see the very nice thesis \cite{ollisaari} for some results concerning the weights arising in doubly nonlinear quasilinear equations).
\end{remark}

\subsection{Function Spaces}\label{function_spaces}

Let $1\leq \vt < \infty$, then $W_0^{1,\vt}(\Om)$ denotes the standard Sobolev space which is the completion of $C_c^{\infty}(\Om)$ under the $\|\cdot\|_{W^{1,\vt}}$ norm. 

The parabolic space $L^{\vt}(-T,T; W^{1,\vt}(\Om))$ for any $\vt \in [1,\infty)$ is the collection of measurable functions $\phi(x,t)$ such that for almost every $t \in (-T,T)$, the function $x \mapsto \phi(x,t)$ belongs to $W^{1,\vt}(\Om)$ with the following norm being finite:
\[
 \| \phi\|_{L^{\vt}(-T,T; W^{1,\vt}(\Om)} := \lbr \int_{-T}^T \| \phi(\cdot, t) \|_{W^{1,\vt}(\Om)}^{\vt} \ dt \rbr^{\frac{1}{\vt}} < \infty.
\]

Analogously, the parabolic space $L^{\vt}(-T,T; W_0^{1,\vt}(\Om))$ is the collection of measurable functions $\phi(x,t)$ such that for almost every $t \in (-T,T)$, the function $x \mapsto \phi(x,t)$ belongs to $W_0^{1,\vt}(\Om)$.

Given a weight $\om \in A_{\vt}$ for some $\vt \in [1,\infty)$, the weighted Lebesgue space $L^{\vt}(-T,T; L^{\vt}_{\om}(\Om))$ is the set of all measurable functions $\phi: \Om_T \mapsto \RR$ satisfying 
\[
 \int_{-T}^T \lbr \int_{\Om} |\phi(x,t)|^{\vt} \om(x,t)\  dx \rbr  dt < \infty.
\]

 Let us recall the following important characterization of Lebesgue spaces:
\begin{lemma}
\label{weight_level_set}
Let $\Om$ be a bounded domain in $\RR^{n}$ and let $w \in L^1(\Om_T)$ be any non-negative function, then for all $\be > \al > 1$ and any non-negative measurable function $g(x,t): \Om_T \mapsto \RR$, there holds
\[
\iint_{\Om_T} g^{\be} w(z) \ dz = \be \int_0^{\infty} \la^{\be -1} w(\{ z \in \Om_T: g(z) > \la\}) \ d\la = (\be -\al) \int_{0}^{\infty} \la^{\be-\al-1} \lbr \iint_{\{z \in \Om_T: g(z) > \la\}} g^{\al} w(z) \ dz \rbr \ d\la.
\]

\end{lemma}

Before we conclude this subsection, let us now recall the well known Poincar\'e's inequality (see \cite[Corollary 8.2.7]{AH} for the proof):
\begin{theorem}
\label{poincare}
Let $1\leq \vt < \infty$ and let $f \in W^{1,\vt}(\tilde{\Om})$ for some bounded domain $\tilde{\Om}$ and suppose that the following measure density condition holds:
\[\abs{\{ x \in \tilde{\Om}: f(x) = 0\}} \geq m_e > 0,\] then there holds
\[
\int_{\tilde{\Om}} \abs{\frac{f}{\diam(\tilde{\Om})}}^{\vt}\ dx \leq C_{(n,\vt,m_e)} \int_{\tilde{\Om}} |\nabla f|^{\vt}  \ dx.
\]

\end{theorem}

 \subsection{Parabolic metric}
 Let us define the Parabolic metric on $\RR^{n+1}$ that will be used throughout the paper:
 \begin{definition}
 \label{parabolic_metric}
 We define the parabolic metric $d_p$ on $\RR^{n+1}$ as follows: Let $z_1 = (x_1,t_1)$ and $z_2 = (x_2,t_2)$ be any two points on $\RR^{n+1}$, then 
 \begin{equation*}
 \label{par_met}
 d_p(z_1,z_2) := \max \mgh{|x_1-x_2|, \sqrt{|t_1-t_2|}}.
 \end{equation*}

 \end{definition}


\subsection{Maximal Function}

For any $f \in L^1(\RR^{n+1})$, let us now define the strong maximal function in $\RR^{n+1}$ as follows:
\begin{equation}
 \label{par_max}
 \mm(|f|)(x,t) := \sup_{\tQ \ni(x,t)} \fiint_{\tQ} |f(y,s)| \ dy \ ds,
\end{equation}
where the supremum is  taken over all parabolic cylinders $\tQ_{a,b}$ with $a,b \in \RR^+$ such that $(x,t)\in \tQ_{a,b}$. An application of the Hardy-Littlewood maximal theorem in $x-$ and $t-$ directions shows that the Hardy-Littlewood maximal theorem still holds for this type of maximal function (see \cite[Lemma 7.9]{Gary} for details):
\begin{lemma}
\label{max_bound}
 If $f \in L^1(\RR^{n+1})$, then for any $\al >0 $, there holds
 \[
  |\{ z \in \RR^{n+1} : \mm(|f|)(z) > \al\}| \leq \frac{5^{n+2}}{\al} \|f\|_{L^1(\RR^{n+1})},
 \]
 and if $f \in L^{\vartheta}(\RR^{n+1})$ for some $1 < \vartheta \leq \infty$, then there holds
 \[
  \| \mm(|f|) \|_{L^{\vartheta}(\RR^{n+1})} \leq C_{(n,\vartheta)} \| f \|_{L^{\vartheta}(\RR^{n+1})}.
 \]

\end{lemma}

\subsection{Notation}

We shall clarify the notation that will be used throughout the paper: 
\begin{enumerate}[(i)]
    
 \item\label{not1} We shall use $\nabla$ to denote derivatives with respect the space variable $x$.
\item\label{not2} We shall sometimes alternate between using $\ddt{f}$, $\pa_t f$ and $f'$ to denote the time derivative of a function $f$.
 \item\label{not3} We shall use $D$ to denote the derivative with respect to both the space variable $x$ and time variable $t$ in $\RR^{n+1}$. 
 \item\label{not4}  Let $z_0 = (x_0,t_0) \in \RR^{n+1}$ be a point and $\rho, s >0$ be two given parameters and let $\la \in (0,\infty)$. We shall use the following notation to denote the following regions:
 \begin{equation*}
\left\{
\begin{array}{rcll}
Q_{\rho}^{\la}(z_0) & := & Q_{\rho,\la^{2-p}\rho^2}(z_0) & \text{for} \ p \geq 2,\\
Q_{\rho}^{\la}(z_0) & := & Q_{\la^{\frac{p-2}{2}}\rho,\rho^2}(z_0) & \text{for} \ p \leq 2,
\end{array}\right.
\end{equation*}


 \begin{equation}\label{notation_space_time}
\def\arraystretch{1.5}
 \begin{array}{ll}
 I_s(t_0) := (t_0 - s, t_0+s) \subset \RR,
& \qquad Q_{\rho,s}(z_0) := B_{\rho}(x_0) \times I_{s}(t_0) \subset \RR^{n+1},\\ 
 \al Q_{\rho,s}(z_0) := B_{\al \rho}(x_0) \times I_{\al^2s}(t_0)  \subset \RR^{n+1},
 &\qquad Q_{\rho}(z_0) := B_{ \rho}(x_0) \times I_{\rho^2}(t_0)  \subset \RR^{n+1},\\
Q_{\rho}^{\la,+}(z_0)  :=  Q_{\rho}^{\la}(z_0) \cap \{ (x,t): x_n >0\}, 
& \qquad Q_{\rho}^{+}(z_0)  :=  Q_{\rho}(z_0) \cap \{ (x,t): x_n >0\}, \\
K_{\rho}^{\la}(z_0)  :=  Q_{\rho}^{\la}(z) \cap \Om_T, & \qquad K_{\rho}(z_0)  :=  Q_{\rho}(z) \cap \Om_T.
 \end{array}
\end{equation}

\item\label{not5} We shall use $\int$ to denote the integral with respect to either space variable or time variable and use $\iint$ to denote the integral with respect to both space and time variables simultaneously. 

Analogously, we will use $\fint$ and $\fiint$ to denote the average integrals as defined below: for any set $A \times B \subset \RR^n \times \RR$, we define
\begin{gather*}
\avgs{f}{A}:= \fint_A f(x) \ dx = \frac{1}{|A|} \int_A f(x) \ dx,\\
\avgs{f}{A\times B}:=\fiint_{A\times B} f(x,t) \ dx \ dt = \frac{1}{|A\times B|} \iint_{A\times B} f(x,t) \ dx \ dt.
\end{gather*}

\item\label{not6} Given any positive function $\mu$, we shall denote $\avgs{f}{\mu} := \int f\frac{\mu}{\|\mu\|_{L^1}}dm$ where the domain of integration is the domain of definition of $\mu$ and $dm$ denotes the associated measure. 


\end{enumerate}

\subsection{Weak solutions}

For this subsection, let us consider the following general problem:
\begin{equation}
\label{ext_u}
\left\{ \begin{array}{rcll}
  \phi_t - \dv \mathcal{A}(z,\nabla \phi) &=& -\dv |\vec{f}|^{p-2}\vec{f} & \quad \text{in} \ \tom_T, \\
  \phi &=& f & \quad \text{on} \  \pa \tom \times (-T,T),\\
  \phi(\cdot,-T) & = & f_0 & \quad \text{on} \ \tom.
 \end{array}\right. 
\end{equation}

Now let us define the Steklov average as follows: let $h \in (0, 2T)$ be any positive number, then we define
\begin{equation}\label{stek1}
  \phi_{h}(\cdot,t) := \left\{ \begin{array}{ll}
                              \hint_t^{t+h} \phi(\cdot, \tau) \ d\tau \quad & t\in (-T,T-h), \\
                              0 & \text{else}.
                             \end{array}\right.
 \end{equation}

\begin{definition}[Weak solution] 
\label{very_weak_solution}
 
 We  then say $\phi \in L^2(-T,T; L^2(\Om)) \cap L^{p}(-T,T; W_0^{1,p}(\Om))$ is a  weak solution of \eqref{main} if the following holds for any $\psi \in W_0^{1,p}(\Om) \cap L^{\infty}(\Om)$:
 \begin{equation}
 \label{def_weak_solution}
  \int_{\Om \times \{t\}} \frac{d [\phi]_{h}}{dt} \psi + \iprod{[\aa(x,t,\nabla \phi)]_{h}}{\nabla \psi} \ dx = \int_{\Om \times \{t\}} \iprod{|\vec{f}|^{p-2}\vec{f}}{\nabla \psi}\ dx \txt{for a.e.}-T < t < T-h.
 \end{equation}
Moreover, the initial datum is taken in the sense of $L^2(\Om)$, i.e.,
\[
\int_{\Om}\abs{\phi_h(x,-T) - f_0(x)}^2\ dx  \xrightarrow{h \searrow 0} 0.
\]

\end{definition}

\label{existence}
We have the following well known existence result (for example, see \cite[Chapter III, Section 6]{showalter2013monotone} for the details):
\begin{proposition}
\label{ext_sol}
Let $\Om$ be any bounded domain satisfying a uniform measure density condition, i.e., there exists a constant $m_e>0$ such that $|B_r(y) \cap \Om| \geq m_e |B_r(y)|$ holds for every $r>0$ and $y \in \pa\Om$ and suppose that $\vec{f} \in L^{p}(\Om_T)$, $\nabla f \in L^{p}(\Om_T)$ with $\ddt{f} \in \lbr W^{1,p}(\Om_T) \rbr'$  and $f_0 \in L^{2}(\Om)$ are given. Then there exists a unique weak solution $\phi \in C^0\lbr-T,T;L^2(\Om)\rbr\cap L^{p}\lbr-T,T; W^{1,p}(\Om)\rbr$ solving \eqref{ext_u}.
%

Moreover if $f =0$, then we have the following energy estimate
\begin{equation*}
\label{energy_phi}
\sup_{-T \leq t \leq T} \|\phi(\cdot,t)\|_{L^2(\Om)}^2 + \iint_{\Om_T} |\nabla \phi|^{p} \ dz \leq C_{(n,p,\lamot)} \lbr \iint_{\Om_T}|\vec{f}|^{p}  \ dz + \|f_0\|_{L^2(\Om)}^2\rbr.
\end{equation*}
\end{proposition}

\subsection{Gradient higher integrability estimates}
In this subsection, let us collect a few important higher integrability results that will be used throughout the paper. In order to state the general theorems, let $\phi \in  L^2(-T,T; L^2(\Om)) \cap L^{p}\lbr -T,T;W_0^{1,p}(\Om)\rbr$ be a weak solution of 
\begin{equation}
 \label{pde_high}
\left\{ \begin{array}{rcll}
  \phi_t - \dv \aa(x,t,\nabla \phi) &=& -\dv (|\vec{f}|^{p-2} \vec{f}) & \quad \text{in} \ \omt,\\
  \phi &=& 0 & \quad \text{on} \ \partial \Om \times (-T,T),
 \end{array}\right. 
\end{equation}
where the nonlinearity is assumed to satisfy \eqref{abounded} and \eqref{bbounded}. Here the domain is assumed to satisfy a uniform measure density condition with constant $m_e$ as in Lemma \ref{measure_density}

The first one is the higher integrability \emph{above the natural exponent}. In the interior case, this was proved in \cite{KL1} whereas in the boundary case, using the measure density condition satisfied by $\Om$, the result was proved in \cite{Mik,Mik1}. 
\begin{lemma}[\cite{Mik,Mik1}]
\label{high_weak}
Let $\tilde{\sigma}>0$ be given, then there exists a $\tilde{\be}_1 = \tilde{\be}_1(n,p,\lamot,m_e) \in (0,\tilde{\sigma}]$ such that if $\vec{f} \in L^{p(1+\tilde{\sigma})}(\Om_T)$ and $\phi\in L^{p}\lbr-T,T;W_0^{1,p}(\Om)\rbr$ is a weak solution to \eqref{pde_high}, then $|\nabla \phi| \in L^{p(1+\be)}(\Om_T)$ for all $\be \in (0,\tilde{\be}_1]$.
Moreover, for any $\mfz \in \overline{\Om} \times (-T,T)$, there holds
\[
\fiint_{K_{\rho}(\mfz)} |\nabla \phi|^{p+\be} \ dz \leq_{(n,p,\lamot,m_e)} \lbr \fiint_{K_{2\rho}(\mfz)}\lbr  |\nabla \phi|+ |\vec{f}|\rbr^{p} \ dz \rbr^{1+\be \tilde{\vt}_1} + \fiint_{K_{2\rho}(\mfz)}\lbr  |\vec{f}|+1\rbr^{p(1+\be)} \ dz.
\]
Here the constant 
\begin{equation*}
\tilde{\vt}_1 :=
\left\{
\begin{array}{ll}
\frac{p}{2} & \txt{if} p \geq 2,\\
\frac{2p}{p(n+2) -2n} &  \txt{if} \frac{2n}{n+2} < p < 2.
\end{array}\right.
\end{equation*}
\end{lemma}

We will also need an improved higher integrability result below the natural exponent. The following theorem was proved for a weaker class of solutions called \emph{very weak solutions}, but also holds true for \emph{weak solutions} as considered in this paper. The interior higher integrability result was proved in the seminal paper \cite{KL} whereas the boundary analogue was proved in \cite{AdiByun2}.
\begin{lemma}[\cite{KL,AdiByun2}]
\label{high_very_weak}
Let  $\vec{f} \in L^{p}(\Om_T)$ and $\phi\in L^{p}\lbr -T,T;W_0^{1,p}(\Om)\rbr$ be the unique  weak solution to \eqref{pde_high}. There exists  $\tilde{\be}_2 = \tilde{\be}_2(n,\lamot,p,m_e) \in (0,1/4)$  such that for any  $\mfz \in \overline{\Om} \times (-T,T)$, there holds
\[
\fiint_{K_{\rho}(\mfz)} |\nabla \phi|^{p} \ dz \leq_{(n,\lamot,p,m_e)} \lbr \fiint_{K_{2\rho}(\mfz)}\lbr  |\nabla \phi|+ |\vec{f}|\rbr^{p-\be} \ dz \rbr^{1+\be \tilde{\vt}_2} + \fiint_{K_{2\rho}(\mfz)}\lbr  |\vec{f}|+1\rbr^{p} \ dz.
\]
Here the constant 
\begin{equation*}
\tilde{\vt}_2 :=
\left\{
\begin{array}{ll}
2-\be & \txt{if} p \geq 2,\\
p-\be -\frac{(2-p)n}{2} &  \txt{if} \frac{2n}{n+2} < p < 2.
\end{array}\right.
\end{equation*}
\end{lemma}

\section{Main Results}
\label{section_two}
In this section, let us describe the main theorem that will be proved. The first is unweighted a priori estimates below the natural exponent.

\begin{theorem}
\label{thm6.1}
Let $\Om$ be a bounded domain satisfying \eqref{msd}, then there exists an $\be_0 = \be_0(p,n,\lamot) \in (0,1)$ such for any $\be \in (0,\be_0)$, the following holds: For any $\bff\in L^{p}(\Om_T)$, let $u \in C^0(-T,T;L^2(\Om)) \cap L^{p}(-T,T; W_0^{1,p}(\Om))$ be the unique \emph{weak solution} of \eqref{main}, then there holds
\[
 \fiint_{\Om_T} |\nabla u|^{p-\be} \ dz \apprle_{(n,p,\be,\lamot)} \fiint_{\Om_T} |\bff|^{p-\be}\ dz.
 \]
\end{theorem}

\begin{remark}As a corollary, we can extend the results of \cite[Theorem 1.6]{duong2017global} to obtain Lorentz space estimates below the natural exponent. The techniques that we develop to prove Theorem \ref{thm6.1} can be used to obtain the  parabolic analogue of \cite[Theorem 1.2]{AP1} for \emph{weak solutions}. In a forthcoming paper, we obtain these results for more general solutions called \emph{very weak solutions}\end{remark}

The second theorem we will prove is the end point weighted estimate. As mentioned in the introduction, the main contribution is the case $q=p$.
\begin{theorem}
\label{main_theorem}
Let $q \in [p,\infty)$ and $w \in A_{\frac{q}{p}}$ be a Muckenhoupt weight, then there exists a positive constants $\vt_0 = \vt_0(\lamot,n,p,\Om)$ and $\ga = \ga(n,\lamot,p,q)$ such that the following holds: Suppose $(\aa, \Om)$ is $(\ga,S_0)$ vanishing for some fixed $S_0>0$, then the problem \eqref{main} has a unique weak solution $u$ satisfying the estimate
\[
\iint_{\Om_T} |\nabla u|^q w(z) \ dz \apprle_{(n,\lamot,p,q,[w]_{\frac{q}{p}},\Om)} \lbr \iint_{\Om_T} |\bff|^q w(z) \ dz + 1 \rbr^{\vt_0}.
\]

\end{theorem}

\section{Construction of test function via Lipschitz truncation}
\label{section_three}
In this section, we will consider the following two problems: Let $\vec{f} \in L^p(\Om_T)$ be given and suppose that $\varphi \in L^2(-T,T;L^2(\Om)) \cap L^p(-T,T;W^{1,p}_0(\Om))$ is a weak solution of 
\begin{equation}
\label{eqone}
 \left\{
 \begin{array}{ll}
\varphi_t - \dv  \aa(x,t,\nabla \varphi) = \dv |\vec{f}|^{p-2} \vec{f}  & \text{in} \ \Om \times (-T,T).
 \end{array}
\right.
\end{equation}
We will extend $\varphi=0$ on $\Om^c \times (-T,T)$, then for any fixed cylinder $Q_{\rho,s}(\mfz) \subset \RR^n \times (-T,T)$, we see from Proposition \ref{ext_sol} that for any $\vec{g} \in L^p(Q_{\rho,s}(\mfz))$, there exists a unique weak solution $\phi \in L^p(I_s(\mft);W^{1,p}(B_{\rho}(\mfx)))$ solving
\begin{equation}
\label{eqtwo}
 \left\{
 \begin{array}{ll}
\phi_t - \dv  \aa(x,t,\nabla \phi) = \dv |\vec{g}|^{p-2} \vec{g} & \text{in}\  Q_{\rho,s}(\mfz),\\
\phi = \varphi & \text{on} \ \partial_p Q_{\rho,s}(\mfz).
 \end{array}
\right.
\end{equation}
From \eqref{eqone}, we see that the condition $\phi = \varphi$ on $\pa_pQ_{\rho,s}(\mfz)$ makes sense.

In Section \ref{section_four}, we  obtain difference estimates below the natural exponent between equations of the form \eqref{eqone} and \eqref{eqtwo}. In order to do this, we need to construct a suitable test function which will be done in this section.

\subsection{A few well known lemmas}\label{few_well_known_lemmas}

 We shall recall the following well known lemmas that will be used throughout this section.  The first one is a standard lemma regarding integral averages (for a proof in this setting, see for example \cite[Chapter 8.2]{bogelein2007thesis} for the details).
\begin{lemma}
\label{time_average}
Let $\la >0$ be any fixed number and suppose  $[\psi]_h(x,t) : = \hint_{t-\la h^2}^{t+\la h^2} \psi(x,\tau) \ d\tau$ for some $\psi \in L^1_{loc}$. Then we have the following properties:
\begin{enumerate}[(i)]
 \item $[\psi]_h \rightarrow \psi$ a.e $(x,t) \in \RR^{n+1}$ as $h \searrow 0$.
 \item $[\psi]_h(x,\cdot)$ is continuous and  bounded  in time for a.e. $x \in \RR^n$.
 \item For any cylinder $Q_{r, \la r^2} \subset \RR^{n+1}$ with $r >0$, there holds
\[
  \fiint_{Q_{r,\ga r^2}} [\psi]_h(x,t) \ dx \ dt \apprle_n \fiint_{Q_{r,\la(r+h)^2}}\psi(x,t) \ dx \ dt. 
 \]
 \item The function $[\psi]_h(x,t)$ is differentiable with respect to $t \in \RR$, moreover $[\psi]_h(x, \cdot) \in C^{1}(\RR)$ for a.e. $x \in \RR^n$. 
 \end{enumerate}
\end{lemma}

 Let us now prove a time localized version of the Parabolic Poincar\'e inequality.
\begin{lemma}
 \label{lemma_crucial_1}
 Let $\psi \in L^{\vt} (-T,T; W^{1,\vt}(\Om))$ with $\vt \in [1,\infty)$  and suppose that $B_{r} \Subset \Om$ be compactly contained ball of radius $r>0$. Let $I \subset (-T,T)$ be a time interval  and $\rho(x,t) \in L^1(B_r \times I)$ be any positive function such that $$\|\rho\|_{L^{\infty}(B_r\times I)} \apprle_n \frac{|B_r\times I|}{\|\rho\|_{L^1(B_r\times I)}}, $$ and $\mu(x) \in C_c^{\infty}(B_r)$ be  such that $\int_{B_r} \mu(x) \ dx = 1$ with $|\mu| \apprle  \frac{1}{r^n}$ and $|\nabla \mu| \apprle  \frac{1}{r^{n+1}}$, then there holds:
 \begin{equation*}
 \begin{array}{ll}
  \fiint_{B_r \times I} \left|\frac{\psi(z)\lsb{\chi}{J} - \avgs{\psi\lsb{\chi}{J}}{\rho}}{r}\right|^{\vt} \ dz & \apprle_{(n,\vt)} \fiint_{B_r \times I} |\nabla \psi|^{\vt}\lsb{\chi}{J} \ dz + \sup_{t_1,t_2 \in I} \left| \frac{\avgs{\psi\lsb{\chi}{J}}{\mu}(t_2) - \avgs{\psi\lsb{\chi}{J}}{\mu}(t_1)}{r} \right|^{\vt},
  \end{array}
 \end{equation*}
where $\avgs{\psi}{\rho}:= \int_{B_r\times I} \psi(z) \frac{\rho(z)}{\|\rho\|_{L^1(B_r\times I)}}\lsb{\chi}{J} \ dz\ $, $\avgs{\psi\lsb{\chi}{J}}{\mu}(t_i) := \int_{B_r} \psi(x,t_i) \mu(x) \lsb{\chi}{J} \ dx$ and $J \Subset (-\infty,\infty)$ be some fixed time-interval. 
\end{lemma}
\begin{proof}
 Let us first consider the case of $\rho(x,t) = \mu(x)\lsb{\chi}{I}(t)$. In this case, we get
 \begin{equation*}
 \label{lem5.1.2}
  \begin{array}{rcl}
  \fiint_{B_r\times I} \left| \frac{\psi(z)\lsb{\chi}{J}  - \avgs{\psi\lsb{\chi}{J}}{\mu \times \lsb{\chi}{I}}}{r} \right|^{\vt} \ dz & \apprle& \fiint_{B_r\times I} \left| \frac{\psi(z)\lsb{\chi}{J}  - \avgs{\psi\lsb{\chi}{J}}{\mu}(t)}{r} \right|^{\vt} + \left| \frac{\avgs{\psi\lsb{\chi}{J}}{\mu}(t)  - \avgs{\psi\lsb{\chi}{J}}{\mu\times {I}}}{r} \right|^{\vt} \ dz \\
 & \overset{\redlabel{lemma2.14.a}{a}}{\apprle} &  \fiint_{B_r\times I} |\nabla \psi|^{\vt}\lsb{\chi}{J} \ dz + \sup_{t_1,t_2 \in {I}} \left|\frac{\avgs{\psi\lsb{\chi}{J}}{\mu}(t_2) - \avgs{\psi\lsb{\chi}{J}}{\mu}(t_1)}{r}\right|^{\vt} \\
 & =  &  \fiint_{B_r\times I} |\nabla \psi|^{\vt}\lsb{\chi}{J} \ dz + \sup_{t_1,t_2 \in {I}\cap J} \left|\frac{\avgs{\psi}{\mu}(t_2) - \avgs{\psi}{\mu}(t_1)}{r}\right|^{\vt}.
  \end{array}
 \end{equation*}
To obtain $\redref{lemma2.14.a}{a}$ above, we made us of the standard Poincar\'e's inequality in the spatial direction which only needs to be applied over a.e. $t \in I \cap J$. Note that the derivative is only in the spatial direction and hence the term $\lsb{\chi}{J}$ does not cause any problem when applying Poincar\'e's inequality.

For the general case, we observe that 
\begin{equation}
  \label{lem5.1.3}
  \fiint_{B_r\times I} \left|\frac{\psi\lsb{\chi}{J} - \avgs{\psi\lsb{\chi}{J}}{\rho}}{r}\right|^{\vt} \ dz \apprle \fiint_{B_r\times I} \left|\frac{\psi\lsb{\chi}{J} - \avgs{\psi\lsb{\chi}{J}}{\mu \times \lsb{\chi}{I}}}{r}\right|^{\vt} \ dz  + \fiint_{B_r\times I} \left|\frac{\avgs{\psi\lsb{\chi}{J}}{\rho} - \avgs{\psi\lsb{\chi}{J}}{\mu \times \lsb{\chi}{I}}}{r}\right|^{\vt} \ dz.
 \end{equation}
The first term of \eqref{lem5.1.3} can be controlled as in \eqref{lem5.1.2} and  to control the second term, we observe that
\begin{equation*}
 \label{lem5.1.4}
 \begin{array}{ll}
 \abs{\avgs{\psi\lsb{\chi}{J}}{\rho} - \avgs{\psi\lsb{\chi}{J}}{\mu \times \lsb{\chi}{I}}} & \leq \frac{\|\rho\|_{L^\infty(B_r\times I)}}{\|\rho\|_{L^1(B_r\times I)}} \iint_{B_r\times I} \abs{\psi\lsb{\chi}{J} - \avgs{\psi\lsb{\chi}{J}}{\mu \times \lsb{\chi}{I}}}\ dz  \apprle \fiint_{B_r\times I} \abs{\psi\lsb{\chi}{J} - \avgs{\psi\lsb{\chi}{J}}{\mu \times \lsb{\chi}{I}}}\ dz.
 \end{array}
\end{equation*}
This completes the proof of the Lemma. 
\end{proof}

\begin{remark}
In Lemma \ref{lemma_crucial_1}, we can take any bounded region $\tilde{\Om}$ instead of $B_r$ such that $\tilde{\Om}$ admits the $\vartheta$-Poincar\'e inequality. 
For example, if $\tilde{\Om}$ satisfies the measure density condition as defined in Definition \ref{measure_density} for some $m_e > 0$, then Lemma \ref{lemma_crucial_1} is applicable.
\end{remark}

We will use the following result  which can be found in \cite[Theorem 3.1]{PG} (see also \cite{Prato}) for proving the Lipschitz regularity for the constructed test function. This very useful simplification of the original technique from \cite{KL} first appeared in \cite[Chapter 3]{bogelein2013regularity}.
\begin{lemma}
\label{metric_lipschitz}
 Let $\ga >0$ and  $\mathcal{D} \subset \RR^{n+1}$ be given. For any $z \in \mathcal{D}$ and $r>0$, let $Q_{r,\ga r^2}(z)$ be the parabolic cylinder centered  at $z$ with radius $r$.  Suppose  there exists a constant $C>0$ independent of $z$ and $r$ such that the following bound holds:
 \[
  \frac{1}{|Q_{r,\ga r^2}(z)\cap \mathcal{D}|} \iint_{Q_{r,\ga r^2}(z)\cap \mathcal{D}} \left|\frac{f(x,t) - \avgs{f}{Q_{r,\ga r^2}(z)\cap \mathcal{D}}}{r}\right| \ dx \ dt \leq C \qquad \forall\  z \in \mathcal{D} \text{ and } r>0,
 \]
 then $f$ is Lipschitz with respect to the metric $d (z_1,z_2) := \max \{ |x_1-x_2|, \sqrt{\ga^{-1} |t_1-t_2| }\}$.
\end{lemma}

\subsection{Construction of test function}

Let us denote the following functions:
\begin{gather*}
v(z) := \varphi(z) - \phi(z) \txt{and} v_h(z) := [\varphi-\phi]_h(z) \label{def_vh}.
\end{gather*}
where $[\cdot]_h$ denotes the usual Steklov average. From Lemma \ref{time_average},  we see that $v_h \xrightarrow{h \searrow 0} v$.  It is easy to see from  \eqref{eqtwo} that $v(z) = 0$ for $z \in \pa_p \ors$. 	

Let us fix the following exponents for this Section:
\begin{equation}
\label{expon_lip_lem}
1 < q < p-2\be<p-\be < p,
\end{equation}
for some $\be \in (0,1)$. Note that eventually we will obtain a $\be_0 = \be_0(n,p,\lamot,m_e)$ such that all the estimates hold for any $\be \in (0,\be_0)$. 

Let us now define the following function:
\begin{equation*}
\label{def_g}
g(z)  := \mm \lbr \lbr[[]|\nabla v|^q + |\nabla \varphi|^q + |\nabla \phi|^q + |\vec{f}|^q + |\vec{g}|^q \rbr[]]\lsb{\chi}{\ors}\rbr^{\frac1{q}}(z),
\end{equation*}
where $\mm$ is as defined in \eqref{par_max}.

For a fixed $\la >0$, let us define the \emph{good set} by 
\begin{equation*}
\label{elambda}
\elam := \{ (x,t) \in \RR^{n+1} : g(x,t) \leq \la\}.
\end{equation*}

We now have the following parabolic Whitney type decomposition of $\elam^c$ (see \cite[Lemma 3.1]{diening2010existence} or \cite[Chapter 3]{bogelein2013regularity} for details):
\begin{lemma}
\label{whitney_decomposition}

Let $\kappa := \la^{2-p}$, then there exists an $\ka$-parabolic Whitney covering  $\{Q_i(z_i)\}$ of $\elam^c$ in the following sense:
 \begin{description}
  \descitem{(W1)}{W1} $Q_j(z_j) = B_j(x_j) \times I_j(t_j)$ where $B_j(x_j) = B_{r_j}(x_j)$ and $I_j(t_j) = (t_j - \ka r_j^2, t_j + \ka r_j^2)$. 
  \descitem{(W2)}{W2} we have $d_{\la}(z_j,\elam) = 16r_j$.
  \descitem{(W3)}{W3} $\bigcup_j \frac12 Q_j(z_j) = \elam^c$.
  \descitem{(W4)}{W4} for all $j \in \NN$, we have $8Q_j \subset \elam^c$ and $16Q_j \cap \elam \neq \emptyset$.
  \descitem{(W5)}{W5} if $Q_j \cap Q_k \neq \emptyset$, then $\frac12 r_k \leq r_j \leq 2r_k$.
  \descitem{(W6)}{W6} $\frac14 Q_j \cap \frac14Q_k = \emptyset$ for all $j \neq k$.
  \descitem{(W7)}{W7} $\sum_j \lsb{\chi}{4Q_j}(z) \leq c(n)$ for all $z \in \elam^c$.
  \end{description}
  Subject to this Whitney covering, we have an associated partition of unity denoted by $\{ \Psi_j\} \in C_c^{\infty}(\RR^{n+1})$ such that the following holds:
  \begin{description}
  \descitem{(W8)}{W8} $\lsb{\chi}{\frac12Q_j} \leq \Psi_j \leq \lsb{\chi}{\frac34Q_j}$.
  \descitem{(W9)}{W9} $\|\Psi_j\|_{\infty} + r_j \| \nabla \Psi_j\|_{\infty} + r_j^2 \| \nabla^2 \Psi_j\|_{\infty} + \la r_j^2 \| \pa_t \Psi_j\|_{\infty} \leq C$.
  \end{description}
  For a fixed $k \in \NN$, let us define 
  \begin{equation*}\label{Ak}A_k := \left\{ j \in \NN: \frac34Q_k \cap \frac34Q_j \neq \emptyset\right\},\end{equation*} then we have
  \begin{description}
  \descitem{(W10)}{W10} Let $i \in \NN$ be given, then $\sum_{j \in A_i} \Psi_j(z) = 1$  for all $z \in \frac34Q_i$.
  \descitem{(W11)}{W11} Let $i \in \NN$ be given and let  $j \in A_i$, then $\max \{ |Q_j|, |Q_i|\} \leq C(n) |Q_j \cap Q_i|.$
  \descitem{(W12)}{W12}  Let $i \in \NN$ be given and let  $j \in A_i$, then $ \max \{ |Q_j|, |Q_i|\} \leq \left|\frac34Q_j \cap \frac34Q_i\right|.$
  \descitem{(W13)}{W13} For any $i \in \NN$, we have $\# A_i \leq c(n)$.
  \descitem{(W14)}{W14} Let $i \in \NN$ be given, then for any $j \in A_i$, we have $\frac34Q_j \subset 4Q_i$.
 \end{description}
\end{lemma}

Now we define the following Lipschitz extension function as follows:
\begin{equation}
\label{lipschitz_extension}
\vlh(z) := v_h(z) - \sum_i \Psi_i(z) (v_h(z) - v_h^i),
\end{equation}
where
\begin{equation}
\label{lipschitz_extension_one}
v_h^i := \left\{ \begin{array}{ll}
                  \frac1{\|\Psi_i\|_{L^1(\frac34Q_i)}}\iint_{\frac34Q_i} v_h(z) \Psi_i (z) \lsb{\chi}{[\mft-s,\mft+s]} \ dz & \text{if} \ \frac34Q_i \subset B_{\rho}(\mfx) \times [\mft-s,\infty), \\
                   0 & \text{else}.
                  \end{array}\right.
\end{equation}

Since $\varphi - \phi = 0$ on $\pa B_{\rho}(\mfx) \times [\mft-s,\mft+s]$, we can switch between $\lsb{\chi}{[\mft-s,\mft+s]}$ and $\lsb{\chi}{\ors}$ without affecting the calculations.

\begin{remark}
\label{rmk1.3.2}
Note that even though $v_h(x,\mft-s) \neq 0$ in general, nevertheless the following initial boundary values are satisfied:
\begin{itemize}
\item The initial condition $(\varphi-\phi)(x,\mft-s) = 0$ is to be understood in the sense 
\[
[\varphi-\phi]_h(\cdot,\mft-s) \xrightarrow{h \searrow 0} 0\  \text{in}\  L^2 (B_{\rho}(\mfx)).
\]
\item For $(x,\mft-s) \in \elam$, we have $\vlh(x,\mft-s) =v_h(x,\mft-s)$.
\item For $(x,\mft-s) \notin \elam$, we have $\vlh(x,\mft-s) = 0$ by using \eqref{lipschitz_extension_one}.
\end{itemize}
\end{remark}

\begin{remark}
\label{rmk1.3.3}
From Lemma \ref{time_average}, we see that $\vlh(z) \xrightarrow{h \searrow 0} \vl(z)$ almost everywhere. 
\end{remark}

We now have the following useful lemma that can be proved just by using the definition of the weak formulation (see for example \cite[Lemma 3.5]{AdiByun2} for details):
\begin{lemma}
\label{lemma_crucial_2}
 Let $\varphi,\phi,\vec{f},\vec{g}$ be as in \eqref{eqone} and \eqref{eqtwo} and   $h \in (0,2s)$. Let $\al(x) \in C_c^{\infty}({B_{\rho}(\mfx)})$ and $\be(t) \in C^{\infty}(\mft-s,\mft+s)$ with $\be(\mft-s) = 0$ be a  non-negative function and $[\cdot]_h$ be the Steklov average as defined in \eqref{time_average}. Then   the following estimate holds for any time interval $(t_1,t_2) \subset (\mft-s,\mft+s)$:
 \begin{equation*}
  \label{lemma_crucial_2_est}
  \begin{array}{rcl}
  |\avgs{v_h\be}{\al} (t_2) - \avgs{v_h\be}{\al}(t_1)| & \leq & C(\La_1,p)\|\nabla \al\|_{L^{\infty}{({B_{\rho}(\mfx)})}} \|\be\|_{L^{\infty}(t_1,t_2)} \iint_{{B_{\rho}(\mfx)} \times (t_1,t_2)} {[|\nabla \phi|^{p-1} + |\nabla \varphi|^{p-1}]_h} \ dz \\
   & &\qquad  +\|\nabla \al\|_{L^{\infty}{({B_{\rho}(\mfx)})}} \|\be\|_{L^{\infty}(t_1,t_2)} \iint_{{B_{\rho}(\mfx)} \times (t_1,t_2)} [|\vec{f}|^{p-1}+|\vec{g}|^{p-1}]_h \ dz \\
   & &\qquad + \|\phi\|_{L^{\infty}{({B_{\rho}(\mfx)})}} \|\varphi'\|_{L^{\infty}(t_1,t_2)} \iint_{{B_{\rho}(\mfx)} \times (t_1,t_2)} |[\varphi-\phi]_h| \ dz.
  \end{array}
 \end{equation*}
\end{lemma}

\subsection{Properties of the test function}

\begin{lemma}
\label{lemma3.7}
 For any $z \in \elam^c$, we have
 \begin{equation}
  \label{bound_v_l_h}
  |\vlh(z)| \apprle_{(n,p,q,\lamot,b_0)} \rho \la.
 \end{equation}
\end{lemma}
\begin{proof}
By construction of the extension in \eqref{lipschitz_extension},  for $z \in \elam^c$, we see that $\vlh(z) = \sum_j \Psi_j(z) v_h^j$ with $v_h^j = 0$ whenever $\frac34Q_j \nsubseteq B_{\rho}(\mfx) \times [\mft-s,\infty)$.

In order to prove the Lemma, making use of \descref{W8}{W8}, we see that \eqref{bound_v_l_h} follows if the following holds:
\begin{equation}
\label{claim_bound}
|v_h^j| \apprle_{(n,p,q,\lamot,b_0)}\rho\la. 
\end{equation}
We shall now proceed with proving \eqref{claim_bound}.    Since we only have to consider the case  $\frac34Q_j \subset B_{\rho}(\mfx) \times [\mft-s,\infty)$, which automatically implies $r_j \apprle \rho$. We now proceed as follows:
\begin{description}[leftmargin=*]
  \item[Case $r_j \geq \rho$:] In this case, we observe that $B_{\rho}(\mfx) \subset 2B_j$ which gives the following sequence of estimates: 
  \begin{equation*}
  \label{est_1}
   \begin{array}{rcl}
    |v_h^j| & \apprle & r_j \fiint_{\frac32Q_j} \left| \frac{[\varphi-\phi]_h(z)}{r_j} \right| \lsb{\chi}{[\mft-s,\mft+s]} \ dz \\
    & \overset{\redlabel{3.13.a}{a}}{\apprle}& \rho \frac{1}{|16I_j|}\int_{16I_j\cap[\mft-s,\mft+s]} \lbr \hint_{16B_j } \left| \frac{[\varphi-\phi]_h(x,t)}{r_j} \right|^q \ dx\rbr^{\frac1{q}} \ dt \\
 & \overset{\redlabel{3.13.b}{b}}{\apprle} &\rho \frac{1}{|16I_j|}\int_{16I_j\cap[\mft-s,\mft+s]} \lbr \hint_{16B_j} \left|\nabla v_h(x,t) \right|^q \lsb{\chi}{[\mft-s,\mft+s]} \ dx\rbr^{\frac1{q}} \ dt \\
&     \overset{\redlabel{3.13.c}{c}}{\apprle} &\rho \la. 
   \end{array}
  \end{equation*}
  To obtain \redref{3.13.a}{a}, we used the fact that $r_j \apprle \rho$ along with H\"older's inequality, to obtain \redref{3.13.b}{b}, we made use of Poincar\'e's inequality and finally to obtain \redref{3.13.c}{c}, we made use of \descref{W4}{W4}.

  \item[Case $\frac34r_{j} \leq \rho$:]   In this case, we gradually enlarge $\frac34Q_i$ until it goes outside $B_{\rho}(\mfx) \times [-s,\infty)$. As a consequence, we have to further consider two subcases, the first where $2^{\tk_1}Q_j$ crosses the lateral boundary first, and the second when $2^{\tk_2} Q_j$ crosses the initial boundary first. 
  
  Let us define the following constant $k_0 := \min\{ \tk_1,\tk_2\}$ where $\tk_1$ and $\tk_2$ satisfy
   \begin{equation}\label{def_k_0}
   \begin{array}{c}
2^{\tk_1 - 1} r_j < \rho \leq 2^{\tk_1} r_j, \\ 
2^{\tk_2 -1} Q_j \subset  B_{\rho}(\mfx) \times [\mft-s,\infty)  \ \text{  but  } \ 2^{\tk_2}Q_j \nsubseteq B_{\rho}(\mfx) \times [\mft-s,\infty).
\end{array}
   \end{equation}
   Note that $k_0$ denotes the first scaling exponent under which either we end up in the situation $r_j \geq 2^{k_0} \rho$ or $2^{k_0}Q_j$ goes outside $B_{\rho}(\mfx) \times [\mft-s,\infty)$. 
   
   Since we only consider the case $\frac34Q_i \subset B_{\rho}(\mfx) \times [\mft-s,\infty)$, using triangle inequality, we get
   \begin{equation}
   \label{est_2_1} 
   \begin{array}{rcl}
    |v_h^j| & {\apprle} & \sum_{m=0}^{k_0 -2} \lbr \avgs{{[\varphi-\phi]_h}\lsb{\chi}{Q_{\rho,s}(\mfz)}}{2^m Q_j} - \avgs{{[\varphi-\phi]_h\lsb{\chi}{Q_{\rho,s}(\mfz)}}}{2^{m+1}Q_j}\rbr + \avgs{{[\varphi-\phi]_h\lsb{\chi}{Q_{\rho,s}(\mfz)}}}{2^{k_0-1}Q_j} \\ 
    &:= & \sum_{m=0}^{k_0-2}S_1^m + S_2.
   \end{array}
  \end{equation}
  
  We shall estimate $S_1^m$ and $S_2$ separately as follows:
  \begin{description}[leftmargin=*]
  \item[Estimate for $S_1^m$:] In this case,  we see that  $2^{m+1}Q_j \subset B_{\rho}(\mfx)\times [\mft-s,\infty)$. Thus applying Lemma \ref{lemma_crucial_1} for any $\mu \in C_c^{\infty}(B_{2^{m+1}r_j}(x_j))$ satisfying $| \mu(x)| \leq \frac{C(n)}{(2^{m+1}r_j)^{n}}$ and $|\nabla \mu(x)| \leq \frac{C(n)}{(2^{m+1}r_j)^{n+1}}$, we get
   \begin{equation}
   \label{S_1_1.1}
    \begin{array}{rcl}
     S_1^m   
     & \apprle &  (2^{m+1}r_j) \lbr \fiint_{2^{m+1}Q_j} |[\nabla (\varphi-\phi)]_h|^q \lsb{\chi}{\ors}\ dz\rbr^{\frac1{q}} \\
     & & \quad + (2^{m+1}r_j) \lbr \sup_{t_1,t_2 \in {2^{m+1}I_j\cap [\mft-s,\mft+s]}} \left|\frac{\avgs{{[\varphi-\phi]_h}}{\mu}(t_2)-\avgs{{[\varphi-\phi]_h}}{\mu}(t_1)}{2^{m+1}r_j} \right|^q\rbr^{\frac1{q}}\\
     & \overset{\text{\descref{W4}{W4}}}{\apprle} & (2^{m+1}r_j) \la + (2^{m+1}r_j) \lbr \sup_{t_1,t_2 \in {2^{m+1}I_j\cap [\mft-s,\mft+s]}} \left|\frac{\avgs{{[\varphi-\phi]_h}}{\mu}(t_2)-\avgs{{[\varphi-\phi]_h}}{\mu}(t_1)}{2^{m+1}r_j} \right|^q\rbr^{\frac1{q}}.
    \end{array}
   \end{equation}
   
   To estimate the second term on the right of \eqref{S_1_1.1}, using $B_{2^{m+1}r_j}(x_j) \subset B_{\rho}(\mfx) \times [\mft-s,\infty)$, we can apply Lemma \ref{lemma_crucial_2} with the test function $\al(x)=\mu(x)$ and $\be(t) = 1$, which gives for any $t_1,t_2 \in \frac34 I_j\cap[\mft-s,\mft+s]$, the estimate   
   \begin{equation}
\label{S_1_2.1}
 \begin{array}{ll}
  |\avgs{{[\varphi-\phi]_h}}{\mu}(t_2)-\avgs{{[\varphi-\phi]_h}}{\mu}(t_1)| 
  & \overset{\redlabel{4.18.a}{a}}{\apprle}  2^{m+1} r_j \lbr \ka \la^{p-1}\rbr = 2^{m+1} r_j \la.
 \end{array}
\end{equation}
To obtain \redref{4.18.a}{a}, we first applied Lemma \ref{lemma_crucial_2} along with  \descref{W1}{W1}, \descref{W4}{W4} and the definition $\ka = \la^{2-p}$.

Substituting \eqref{S_1_2.1} into \eqref{S_1_1.1}, we get 
\begin{equation}
\label{bound_S_1_m}
S_1^m \apprle 2^{m+1}r_j \la.
\end{equation}

  \item[Estimate for $S_2$:] For this term, we know that $2^{k_0-1}Q_j \notin B_{\rho}(\mfx) \times [\mft-s,\infty)$, which implies $2^{k_0-1}Q_j$ crosses either the lateral boundary $\pa B_{\rho}(\mfx) \times [\mft-s,\infty)$ or crosses the initial boundary $B_{\rho}(\mfx) \times \{\mft-s\}$ first. We will consider both the cases separately and estimate $S_2$ as follows:

  \emph{In the case $2^{k_0-1}Q_j$ crosses the lateral boundary $\pa B_{\rho}(\mfx) \times [\mft-s,\infty)$ first}, we can directly apply Theorem \ref{poincare} to obtain 
  \begin{equation}
\label{second_case.1}
          \fiint_{2^{k_0-1}Q_j} [\varphi-\phi]_h\lsb{\chi}{\ors} \ dz  \apprle ( 2^{k_0} r_j)\lbr \fiint_{ 2^{k_0}Q_j} |\nabla [\varphi-\phi]_h|^q \lsb{\chi}{\ors} \ dz \rbr^{1/q} 
           \overset{\redlabel{4.20.a}{a}}{\apprle} \rho \la.
\end{equation}
   To obtain \redref{4.20.a}{a}, we made use of \descref{W4}{W4} along with $2^{k_0-2}r_j \leq \rho$ given by \eqref{def_k_0}.

 \emph{In the case $2^{k_0} Q_j$ crosses the initial boundary $B_{\rho}(\mfx) \times \{\mft-s\}$ first}, by enlarging the cylinder to $2^{k_1+1}Q_j$, we can find a cut-off function $\theta(x,t)$ such that $\spt \theta(x,t) \subset 2^{k_1+1}Q_j \cap \RR^n \times (-\infty,\mft-s)$, which combined with the fact $v_h(z) \lsb{\chi}{[\mft-s,\mft+s]} = 0$ on $\RR^n \times (-\infty,\mft-s)$, we get $\avgs{v_h\lsb{\chi}{[\mft-s,\mft+s]}}{\theta}=0$. Thus applying Lemma \ref{lemma_crucial_1}, we get
\begin{equation}
\label{second_case}
     \begin{array}{rcl}
          \fiint_{2^{k_0+1}Q_j} |v_h (z)|\lsb{\chi}{[\mft-s,\mft+s]} \ dz\  & = &\fiint_{ 2^{k_0+1}Q_j} \abs{v_h(z)\lsb{\chi}{[\mft-s,\mft+s]} - \avgs{v_h\lsb{\chi}{[\mft-s,\mft+s]}}{\theta}} \ dz\\ 
     & \apprle&  (2^{k_0+1}r_j) \lbr \fiint_{2^{k_0+1}Q_j} |[\nabla (\varphi-\phi)]_h|^q  \lsb{\chi}{[\mft-s,\mft+s]}\ dz  \rbr^{\frac1{q}}  \\
     &  & + (2^{k_0+1}r_j)\lbr  \sup_{t_1,t_2 \in {2^{k_0+1}I_j\cap [\mft-s,\mft+s]}} \left|\frac{\avgs{{[\varphi-\phi]_h }}{\mu}(t_2)-\avgs{{[\varphi-\phi]_h }}{\mu}(t_1)}{2^{k_0+1}r_j} \right|^q\rbr^{\frac1{q}}\\
     & \overset{\redlabel{4.21.a}{a}}{\apprle}& 2^{k_0+1}r_j \la \overset{\redlabel{4.21.b}{b}}{\apprle} \rho \la.
     \end{array}
\end{equation}
To obtain \redref{4.21.a}{a}, we made use of \descref{W1}{W1},\descref{W4}{W4} along with an application of Lemma \ref{lemma_crucial_2} and to obtain \redref{4.21.b}{b}, we used \eqref{def_k_0}.

Combining \eqref{second_case.1} and \eqref{second_case}, we get
\begin{equation}
\label{bound_S_2}
S_2 \apprle \rho \la. 
\end{equation}

\end{description}

Thus combining \eqref{bound_S_1_m} and \eqref{bound_S_2} into \eqref{est_2_1}, we get
\begin{equation*}
|v_h^j| \leq \sum_{m=0}^{k_0-2} S_1^m + S_2  \apprle  \la \lbr \sum_{m=0}^{k_0-2} 2^{m+1}r_j + \rho\rbr \overset{\eqref{def_k_0}}{\apprle} \rho \la.
\end{equation*}
\end{description}
This completes the proof of the Lemma. 
\end{proof}

Now we prove a sharper estimate. 
\begin{lemma}
\label{lemma3.8}
 For any $j \in A_i$, there holds
 \begin{equation*}
  \label{3.26}
  |v_h^i - v_h^j| \apprle_{(n,p,q,\lamot,m_e)} \min\{\rho, r_i\} \la.
 \end{equation*}
\end{lemma}
\begin{proof}
We only have to consider the case $r_i \leq \rho$ because if $\rho \leq r_i$, we can directly use Lemma \ref{lemma3.7} to get the required conclusion.

\begin{description}[leftmargin=*]
\item[If either $v_h^i = 0$ or $v_h^j = 0$,]  then $\frac34Q_i$ must necessarily intersect the lateral or initial boundary. 

     \emph{Initial Boundary Case $\frac34Q_i \subset B_{\rho}(\mfx) \times \RR$:} Without loss of generality, we can assume $2Q_i \subset B_{\rho}(\mfx) \times \RR$. We now pick  $\theta(x,t)\in C_c^{\infty}(\RR^{n+1})$ such that $\spt(\theta) \subset 2B_i \times (-\infty,\mft-s)$. Since $\varphi-\phi = 0$ on $2B_i \times (-\infty,\mft-s)$, we see that $\avgs{{v_h\lsb{\chi}{[\mft-s,\mft+s]}}}{\theta}=\avgs{{[\varphi-\phi]_h\lsb{\chi}{[\mft-s,\mft+s]}}}{\theta}=0$. Thus we get
     \begin{equation*}\label{lemma3.8.0}
          \begin{array}{rcl}
               |v_h^i| & \apprle & \fiint_{2Q_i} \abs{[\varphi-\phi]_h\lsb{\chi}{[\mft-s,\mft+s]}- \avgs{{[\varphi-\phi]_h\lsb{\chi}{[\mft-s,\mft+s]}}}{\theta}} \ dz \\
               & \overset{\redlabel{3.80.a}{a}}{\apprle} & r_i \lbr \fiint_{2Q_i} |\nabla v_h|^q \lsb{\chi}{[\mft-s,\mft+s]}\ dz + \sup_{t_1,t_2 \in 2I_i\cap [\mft-s,\mft+s]} \left| \frac{\avgs{{v_h\lsb{\chi}{[\mft-s,\mft+s]}}}{\mu}(t_2)-\avgs{{v_h\lsb{\chi}{[\mft-s,\mft+s]}}}{\mu}(t_1)}{r_i}\right|^q \rbr^{\frac{1}{q}} \\
& \overset{\redlabel{3.80.b}{b}}{\apprle} & r_i \la. 
          \end{array}
     \end{equation*}
To obtain \redref{3.80.a}{a}, we made use of Lemma \ref{lemma_crucial_1} and to obtain \redref{3.80.b}{b}, we proceed similarly to how \eqref{S_1_2.1} was estimated. 

     \emph{Lateral Boundary Case $\frac34Q_i \cap  (B_{\rho}(\mfx)\times \RR)^c \neq \emptyset$:} In this case, using  Theorem \ref{poincare} and \descref{W4}{W4}, we get
     \begin{equation}
\label{lemma3.8.1}
 \begin{array}{ll}
  |v_h^i| & \apprle r_i \lbr \fiint_{2Q_i} \left| \frac{[\varphi-\phi]_h \lsb{\chi}{[\mft-s,\mft+s]}}{r_i} \right|^q \ dz \rbr^{\frac{1}{q}} \apprle r_i \lbr \fiint_{2Q_i} \left| {\nabla [\varphi-\phi]_h} \right|^q\lsb{\chi}{[\mft-s,\mft+s]} \ dz \rbr^{\frac{1}{q}} \apprle r_i \la. 
 \end{array}
\end{equation}
From  \eqref{lemma3.8.1} and \eqref{claim_bound}, we see that the lemma is proved  provided $v_h^j=0$.

\item[Now let us consider the case $v_h^i \neq 0$ and $v_h^j \neq 0$,] which implies $\frac34Q_i \subset B_{\rho}(\mfx)\times [-s,\infty)$ and $\frac34Q_j \subset  B\times [\mft-s,\infty)$. From the definition of $v_h^i$ in \eqref{lipschitz_extension_one}, triangle inequality and \descref{W12}{W12}, we get
\begin{equation}
 \label{3.29}
 \begin{array}{rl}
 |v_h^i - v_h^j| 
 & \apprle \frac{|\frac34Q_i|}{|\frac34Q_i \cap \frac34Q_j|}\fiint_{\frac34Q_i} \abs{v_h(z)\lsb{\chi}{[\mft-s,\mft+s]} - v_h^i} \ dz + \frac{|\frac34Q_j|}{|\frac34Q_i \cap \frac34Q_j|}\fiint_{ \frac34Q_j}\abs{v_h(z)\lsb{\chi}{[\mft-s,\mft+s]} - v_h^j} \ dz \\
 & \apprle \fiint_{\frac34Q_i} \abs{v_h(z)\lsb{\chi}{[\mft-s,\mft+s]} - v_h^i} \ dz + \fiint_{ \frac34Q_j} \abs{v_h(z)\lsb{\chi}{[\mft-s,\mft+s]} - v_h^j} \ dz.
 \end{array}
\end{equation}

Let us now estimate each of the terms in \eqref{3.29} as follows: we apply H\"older's inequality followed by  Lemma \ref{lemma_crucial_1} with $\al \in C_c^{\infty}\lbr\frac34B_i\rbr$ with $|\al(x)| \apprle \frac{1}{r_i^n}$ and $|\nabla \al(x)| \apprle \frac{1}{r_i^{n+1}}$ to get
\begin{equation}
 \label{3.30}
 \begin{array}{rcl}
 \fiint_{\frac34Q_i} | v_h(z)\lsb{\chi}{[\mft-s,\mft+s]} - v_h^i| \ dz 
 & = &  r_i \lbr \fiint_{\frac34Q_i} |\nabla v_h|^q \lsb{\chi}{[\mft-s,\mft+s]}\ dz \rbr^{\frac{1}{q}} \\
 && + r_i \lbr\sup_{t_1,t_2 \in \frac34I_i\cap[\mft-s,\mft+s]} \left|\frac{\avgs{{[\varphi-\phi]_h}}{\mu}(t_2) - \avgs{{[\varphi-\phi]_h}}{\mu}(t_1)}{r_i} \right|^q \rbr^{\frac{1}{q}}.
 \end{array}
\end{equation}
The first term on the right of \eqref{3.30} can be controlled using \descref{W4}{W4} and the second term can be controlled similarly as \eqref{S_1_2.1}. Thus we get
\[
\fiint_{\frac34Q_i} | v_h(z)\lsb{\chi}{[\mft-s,\mft+s]} - v_h^i| \ dz\apprle r_i \la.
\]
\end{description}
This completes the proof of the Lemma.
\end{proof}

Once we have the bounds in Lemma \ref{lemma3.7} and Lemma \ref{lemma3.8}, we can obtain the following important estimates:

\begin{lemma}
\label{lemma3.9}
 Given any  $z \in \elam^c$, we have $z \in \frac34Q_i$ for some $i \in \NN$. Then there holds
 \begin{equation}
  \label{3.34}
  |\nabla \vlh(z)| \leq C_{(n,p,q,\lamot,m_e)} \la.
 \end{equation}

\end{lemma}

\begin{proof}
  We observe that $\sum_{j} \Psi_j(z) = \sum_{j: j\in A_i} \Psi_j(z)=1$ for any $z \in \elam^c$, which implies
  $\sum_j \nabla \Psi_j(z) =   0$  for all $z \in \elam^c$. Thus using \eqref{lipschitz_extension}  along with \descref{W9}{W9}, \descref{W13}{W13} and Lemma \ref{lemma3.8}, we get
  \begin{equation*}
   \label{3.36}
   \begin{array}{ll}
   |\nabla \vlh(z)| \leq   \sum_{j: j\in A_i} |\nabla \Psi_j(z)| \abs{ v_h^j - v_h^i} \apprle \la.
   \end{array} 
  \end{equation*}
 This completes the proof of the Lemma. 
\end{proof}

\subsection{Estimates on the derivative of \texorpdfstring{$\vlh$}.}
We will now mention some improved estimates which can be proved using H\"older's inequality along with the techniques from Lemma \ref{lemma3.9}.
\begin{lemma}
 \label{lemma3.10.1}
 Let $z \in \elam^c$ and $\ve \in (0,1]$ be any number, then $z \in \frac34Q_i$ for some $i \in \NN$ from \descref{W1}{W1}. There exists a constant $C =  C_{(n,p,q,\lamot,m_e)}$ such that the following holds:
 \begin{gather*}
  |\vlh(z)|  \leq C \fiint_{4Q_i}|v_h(\tz)| \lsb{\chi}{[\mft-s,\mft+s]}\ d\tz \leq  \frac{Cr_i\la}{\varepsilon} + \frac{C\varepsilon}{\la r_i} \fiint_{4Q_i}|v_h(\tz)|^2 \lsb{\chi}{[\mft-s,\mft+s]}\ d\tz, \label{lemma3.10_bound1}\\
  |\nabla \vlh(z)| \leq C \frac{1}{r_i} \fiint_{4Q_i}|v_h(\tz)| \lsb{\chi}{[\mft-s,\mft+s]}\ d\tz \leq  \frac{C \la}{\varepsilon} + \frac{C\varepsilon}{\la r_i^2} \fiint_{4Q_i}|v_h(\tz)|^2 \lsb{\chi}{[\mft-s,\mft+s]}\ d\tz.\label{lemma3.10_bound2}
 \end{gather*}
\end{lemma}

\begin{lemma}
 \label{lemma3.10.2}
 Let $z \in \elam^c$ and $\ve \in (0,1]$ be any number, then $z \in \frac34Q_i$ for some $i \in \NN$ from \descref{W1}{W1}. There exists a constant $C =  C_{(n,p,q,\lamot,m_e)}$ such that the following holds:
 \begin{gather}
  |\vlh(z)| \leq C \lbr \min\{ \rho, r_i\} \la + |v_h^i| \rbr \leq  C \lbr \frac{ r_i\la}{\ve} + \frac{\ve}{r_i \la} |v_h^i|^2 \rbr,  \label{lemma3.10_bound3}\\
  |\nabla \vlh(z)| \leq C \frac{\la}{\ve},\label{lemma3.10_bound4}\\
   |\pa_t \vlh(z)|  \leq C\frac{1}{\la^{2-p} r_i^2} \fiint_{4Q_i} |v_h(\tz)| \lsb{\chi}{[\mft-s,\mft+s]}\ d\tz, \nonumber \\
 |\pa_t \vlh(z)|  \leq C \frac{1}{\la^{2-p} r_i^2} \min\{r_i,\rho\} \la. \label{lemma3.11.bound2}
 \end{gather}
\end{lemma}


\subsection{Some more properties of \texorpdfstring{$\vlh$}.} 

\begin{lemma}
\label{lemma3.12}
For any $\vartheta \geq 1$,  we have the following bound:
 \begin{equation*}
  \label{3.56}
  \iint_{\ors\setminus\elam } |\vlh(z)|^{\vartheta} \ dz \apprle_{(n,p,q,\lamot,m_e)} \iint_{\ors\setminus\elam } |v_h(z)|^{\vartheta} \lsb{\chi}{[\mft-s,\mft+s]}\ dz.
 \end{equation*}
\end{lemma}

\begin{proof}
 Since $\elam^c$ is covered by Whitney cylinders (see Lemma \ref{whitney_decomposition}), let us pick some $i \in \NN$ and consider the corresponding parabolic Whitney cylinder. Using the construction from \eqref{lipschitz_extension} along with \descref{W5}{W5}, \descref{W9}{W9} and \descref{W13}{W13}, we get
 \begin{equation}
 \label{3.57}
   \iint_{\frac34Q_i} |\vlh(z)|^{\vartheta} \ dz  \apprle \sum_{j:j \in A_i} \iint_{\frac34Q_i} \Psi_j(z)^{\vartheta} |v_h^j|^{\vartheta} \ dz
    \apprle \iint_{4Q_i} |v_h(z)|^{\vartheta} \lsb{\chi}{[\mft-s,\mft+s]}\ dz. 
 \end{equation}

 Summing \eqref{3.57} over all $i \in \NN$ and making use of  \descref{W4}{W4} and \descref{W7}{W7}, we get
\begin{equation*}
\begin{array}{ll}
 \iint_{\ors	\setminus \elam} |\vlh(z) |^{\vartheta} \ dz 
  \apprle \sum_i \iint_{4Q_i } |v_h(z)|^{\vartheta} \lsb{\chi}{[\mft-s,\mft+s]} \ dz 
  \apprle \iint_{\ors\setminus \elam} |v_h(z)|^{\vartheta} \lsb{\chi}{[\mft-s,\mft+s]}\ dz.
 \end{array}
\end{equation*}
This proves the Lemma.
\end{proof}

\begin{lemma}
\label{lemma3.14}
 For any $0< \vartheta \leq q$ with $q$ defined as in \ref{expon_lip_lem}, there holds 
 \begin{equation*}
  \iint_{\ors \setminus \elam} |\pa_t \vlh(z)  (\vlh(z) - v_h(z))|^{\vartheta} \ dz \apprle_{(n,p,q,\lamot,m_e, \vartheta)} \la^{\vartheta p} |\RR^{n+1} \setminus \elam|.
 \end{equation*}
\end{lemma}

\begin{proof}
From \descref{W3}{W3}, we see that 
 $\ors \setminus \elam\subset \bigcup_{i \in \ZZ} 4Q_i$, thus, for a given $i \in \NN$, let us define the following: 
 \begin{equation*}
  J_i:= \iint_{\frac34Q_i} |\pa_t \vlh(z)  (\vlh(z) - v_h(z))|^{\vartheta} \lsb{\chi}{\ors}\ dz .
 \end{equation*}

  Making use of  \eqref{lemma3.11.bound2} and H\"older's inequality (recall $\ga = \la^{2-p}$), we get
  \begin{equation}
  \label{3.66}
   \begin{array}{rcl}
J_i  &\apprle & \lbr \frac{1}{\la^{2-p} r_i^2} r_i \la \rbr^{\vartheta} \iint_{\frac34Q_i} |\vlh(z)\lsb{\chi}{\ors} - v_h(z)\lsb{\chi}{\ors}|^{\vartheta} \ dz \\
& \overset{\redlabel{4.42.a}{a}}{\apprle}& \lbr \frac{1}{\la^{2-p} r_i^2} r_i \la \rbr^{\vartheta} \sum_{j \in A_i}\iint_{\frac34Q_i}  |v_h(z)\lsb{\chi}{\ors} - v_h^j|^{\vartheta} \ dz \\
 & \overset{\redlabel{4.42.b}{b}}{\apprle} & \la^{\vartheta p} \ |\frac34Q_i|.
   \end{array}
  \end{equation}
To obtain \redref{4.42.a}{a}, we made use of \eqref{lipschitz_extension}, \descref{W9}{W9} and \descref{W10}{W10} and to obtain \redref{4.42.b}{b}, we applied Theorem \ref{poincare} along with \descref{W4}{W4}.

Summing \eqref{3.66} over all $i \in \NN$ and making use of \descref{W7}{W7} completes the proof of the lemma.
\end{proof}

\subsection{Proof of the Lipschitz continuity of \texorpdfstring{$\vlh$}.}
We shall now prove the Lipschitz continuity of $\vlh$ on $\mch := \RR^n \times [\mft-s,\mft+s]$. 

\begin{lemma}
\label{lemma3.15}
The function $\vlh$ from  \eqref{lipschitz_extension} is $C^{0,1}(\mch)$ with respect to the parabolic metric given in Definition \eqref{par_met}.
\end{lemma}

\begin{proof} 
Let us consider a parabolic cylinder $Q_{r}(z) := Q_{r, \ka r^2} (z) := Q$ for some $z \in \mch$ and $r>0$ (recall $\ka = \la^{2-p}$).  To prove the Lemma, we make use of Lemma \ref{metric_lipschitz} and prove the following bound:
 \begin{equation*}
 \label{bound_I_r}
I_r(z) :=   \fiint_{Q \cap \mch} \left|\frac{\vlh(\tz) - \avgs{\vlh}{Q\cap\mch}}{r}\right| \ d\tz \leq {o}(1),
 \end{equation*}
where ${o}(1)$ denotes a constant independent of $z \in \mch$ and $r>0$ only. We will split the proof into several subcases and proceed as follows:

\begin{description}
\item[Case $2Q \subset \elam^c$:]  In this case, from \descref{W3}{W3}, we see that $z \in \frac34Q_i$ for some $i \in \NN$. From the construction in \eqref{lipschitz_extension}, we see that $\vlh\in C^{\infty}(\elam^c)$ which combined with the mean value theorem gives
  \begin{equation*}
   \label{3.71}
I_r(z)  \apprle   \frac{1}{r} \fiint_{Q \cap \mch} \fiint_{Q \cap \mch}  |\vlh(\tz_1) - \vlh(\tz_2)| \ d\tz_1 \ d\tz_2
 \apprle  \sup_{\tz \in Q \cap \mch}\lbr  |\nabla\vlh(\tz)| + \la^{2-p} r |\pa_t \vlh(\tz)|\rbr .
  \end{equation*}
  Let us pick some  $\tz_0 \in 2Q \subset \elam^c$,  then $\tz_0 \in Q_j$ for some $j \in \NN$. Thus we can make use of \eqref{3.34} and \eqref{lemma3.11.bound2} to get
  \begin{equation}
  \label{3.72}
   |\nabla\vlh(\tz_0)| + \la^{2-p} r |\pa_t \vlh(\tz_0)| \apprle \la + \la^{2-p} r \frac{1}{\la^{2-p} r_j^2} r_j \la.
  \end{equation}

  In \eqref{3.72}, we need to understand the relation between $r_j$ and $r$. To this end, from $2Q \subset \elam^c$, we see that 
\begin{equation}
 \label{3.74}
 r \leq  d_{\la} (\tz_0, \elam) \leq d_{\la} (\tz_0, z_j) + d_{\la} (z_j , \elam) \leq r_j + 16r_j = 17r_i. 
\end{equation}

Combining \eqref{3.72} and \eqref{3.74}, we get
\begin{equation*}
\label{3.73}
|\nabla\vlh(\tz_0)| + \la^{2-p} r |\pa_t \vlh(\tz_0)| \apprle \la.
\end{equation*}

\item[Case $2Q \nsubseteq \elam^c$:] In this case, we shall split the proof into three subcases:
\begin{description}[leftmargin=*]
\item[Subcase $2Q \subset \RR^n \times {(-\infty,\mft+s]}$ or  $2Q \subset \RR^n \times [\mft-s,\infty)$:]
In this situation, it is easy to see that the following holds:
\begin{equation}
  \label{3.70}
  |Q \cap \mch| \apprge |Q|. 
 \end{equation}
 We apply   triangle inequality and estimate $I_r(z)$ by 
 \begin{equation}
  \label{3.81}
  \begin{array}{ll}
   I_r(z) & \leq \fiint_{Q \cap \mch} \left| \frac{\vlh(\tz) - v_h(\tz)}{r}\right|+ \left| \frac{v_h(\tz) - \avgs{v_h}{Q\cap \mch}}{r}\right| + \left| \frac{\avgs{v_h}{Q\cap \mch} - \avgs{\vlh}{Q \cap \mch}}{r}\right| \ d\tz \\
   & \leq 2J_1 + J_2,
  \end{array}
 \end{equation}
 where we have set
 \begin{gather}
  J_1:= \fiint_{Q \cap \mch} \left| \frac{\vlh(\tz) - v_h(\tz)}{r}\right| \ d\tz  \txt{and} J_2 := \fiint_{Q \cap \mch} \left| \frac{v_h(\tz) - \avgs{v_h}{Q\cap \mch}}{r}\right|\ d\tz.\label{def_J_1_2}
 \end{gather}
We now estimate each of the terms of \eqref{def_J_1_2} as follows:
\begin{description}[leftmargin=*]
\item[Estimate for $J_1$:]  From \eqref{lipschitz_extension}, we get
\begin{equation}
\label{3.82}
\begin{array}{ll}
J_1 &\apprle \sum_{i\in \NN} \frac{1}{|Q\cap\mch|} \iint_{Q \cap\mch\cap \frac34Q_i} \left| \frac{v_h(\tz)\lsb{\chi}{[\mft-s,\mft+s]} - v_h^i}{r}\right| \ d\tz.
\end{array}
\end{equation}
Let us fix an $i \in \NN$ and take two points $\tz_1 \in Q \cap \frac34Q_i$ and $\tz_2 \in \elam \cap 2Q$. Let $z_i$ denote the center of $\frac34Q_i$,  making use of  \descref{W2}{W2} along with the trivial bound $ d_{\la}(\tz_1, \tz_2) \leq  4r$ and $d_{\la} (z_i, \tz_1) \leq 2r_i$,  we get
\begin{equation}
\label{3.83.1}
16r_i =d_{\la}(z_i,\elam) \leq d_{\la} (z_i, \tz_1) + d_{\la} (\tz_1, \tz_2) \leq 2r_i + 4r  \ \Longrightarrow \ 2r_i \leq r.
\end{equation}

Note that \eqref{3.70} holds and thus summing over all $i \in \NN$ such that  $Q \cap\mch\cap \frac34Q_i \neq \emptyset$ in \eqref{3.82} and making use of  \eqref{3.83.1}, we get
\begin{equation*}
 \label{bound_J_1}
 \begin{array}{rcl}
 J_1 &\apprle& \sum_{\substack{i\in\NN \\ Q \cap\mch\cap \frac34Q_i \neq \emptyset }} \frac{|\frac34Q_i|}{|Q\cap\mch|} \fiint_{\frac34Q_i} \left| \frac{v_h(\tz)\lsb{\chi}{[\mft-s,\mft+s]} - v_h^i}{r}\right| \ d\tz\\ 
 & \overset{\redlabel{4.53.a}{a}}{\apprle} & \sum_{i\in \NN}  \fiint_{\frac34Q_i} \left| \frac{v_h(\tz)\lsb{\chi}{[\mft-s,\mft+s]} - v_h^i}{r_i}\right| \ d\tz \\
 & \overset{\redlabel{4.53.b}{b}}{\apprle} & \la.
 \end{array}
\end{equation*}
To obtain \redref{4.53.a}{a}, we made use of \eqref{3.70} and \eqref{3.83.1}, to obtain \redref{4.53.b}{b}, we follow the calculation from bounding \eqref{3.30}.

\item[Estimate for $J_2$:] Note that $Q \cap \mch$ is another cylinder. \emph{If $Q \subset B_{\rho}(\mfx) \times \RR$,} then choose a cut-off function $\mu \in C_c^{\infty}(B_{\rho}(\mfx))$ and apply Lemma \ref{lemma_crucial_1} to get
\begin{equation*}
\begin{array}{rcl}
J_2  
& \apprle & \lbr \fiint_{Q \cap \mch} |\nabla v_h|^q \lsb{\chi}{\ors} + \sup_{t_1,t_2 \in [\mft-s,\mft+s] \cap Q} \left| \frac{\avgs{v_h\lsb{\chi}{\ors}}{\mu}(t_1) - \avgs{v_h\lsb{\chi}{\ors}}{\mu}(t_1)}{r} \right|^q\rbr^{\frac{1}{q}}.
\end{array}
\end{equation*}
Recall that we are in the case $2Q \cap \elam \neq \emptyset$ and $2Q \cap \elam^c \neq \emptyset$. Further applying Lemma \ref{lemma_crucial_2} and proceeding as in \eqref{S_1_1.1}, we get
\begin{equation}
\label{J_2_bound_first_case}
J_2 \apprle \la. 
\end{equation}

\emph{On the other hand, if $Q \nsubseteq B_{\rho}(\mfx) \times \RR$,} then we can apply Poincar\'e's inequality from Theorem \ref{poincare} directly and make use of the fact that $2Q \cap \elam \neq \emptyset$ to get
\begin{equation*}
\label{bound_J_2_second_case}
J_2 
 \apprle \lbr \fiint_{Q \cap \mch} \left| \nabla v_h(\tz)\lsb{\chi}{[\mft-s,\mft+s]}\right|^q\ d\tz\rbr^{\frac{1}{q}} \apprle \la.
\end{equation*}

\end{description}

\item[Subcase $2Q \cap \RR^n \times {(-\infty,\mft-s]} \neq \emptyset$ and $2Q \cap \RR^n \times [\mft+s,\infty)\neq \emptyset$ AND $\ka r^2 \leq s$:] In this case, we see that $$|Q \cap \mch| = |B_1|r^n \times 2s.$$
 We apply   triangle inequality and estimate $I_r(z)$ as we did in \eqref{3.81} to get
 \begin{equation*}
  \label{3.81_n}
  \begin{array}{ll}
   I_r(z) 
   & \leq 2J_1 + J_2,
  \end{array}
 \end{equation*}
 where we have set
 \begin{gather*}
  J_1:= \fiint_{Q \cap \mch} \left| \frac{\vlh(\tz) - v_h(\tz)}{r}\right| \ d\tz  \txt{and} J_2 := \fiint_{Q \cap \mch} \left| \frac{v_h(\tz) - \avgs{v_h}{Q\cap \mch}}{r}\right|\ d\tz.\label{def_J_1_2_n}
 \end{gather*}
 
 We estimate $J_1$ as follows
 \begin{equation*}
 \label{3.82.n}
 \begin{array}{rcl}
 J_1 & \apprle &\sum_{i \in \NN} \frac{|\frac34Q_i|}{|Q\cap\mch|} \fiint_{\frac34Q_i} \left| \frac{v_h(\tz)\lsb{\chi}{[\mft-s,\mft+s]} - v_h^i}{r}\right| \ d\tz \\
 & \overset{\eqref{3.83.1}}{\apprle}& \frac{r_i^{n+2} \ka}{r^{n} s} \sum_{i\in \NN}  \fiint_{\frac34Q_i} \left| \frac{v_h(\tz)\lsb{\chi}{[\mft-s,\mft+s]} - v_h^i}{r_i}\right| \ d\tz\\
 & \overset{\eqref{3.83.1}}{\apprle}& \frac{r^{n+2} \ka}{r^{n} s} \sum_{i\in \NN}  \fiint_{\frac34Q_i} \left| \frac{v_h(\tz)\lsb{\chi}{[\mft-s,\mft+s]} - v_h^i}{r_i}\right| \ d\tz\\
 & \overset{\redlabel{4.58.a}{a}}{\apprle} &\frac{r^{2} \ka}{ s} \la\\
 & \overset{\redlabel{4.58.b}{b}}{\apprle} & \la.
 \end{array}
 \end{equation*}
To obtain \redref{4.58.a}{a}, we proceed similarly to \eqref{3.30} and to obtain \redref{4.58.b}{b}, we made use of  $\ka r^2 \leq s$.

 The estimate for $J_2$ is already obtained in \eqref{J_2_bound_first_case} which shows
 \begin{equation*}
 \label{bound_J_2_third_case}
 J_2 \apprle \la. 
 \end{equation*}
 
 \item[Subcase $2Q \cap \RR^n \times {(-\infty,\mft-s]} \neq \emptyset$ and $2Q \cap \RR^n \times [\mft+s,\infty)\neq \emptyset$ AND $\ka r^2 > s$:] Using triangle inequality and the bound $|Q \cap \mch| = |B_1|r^n \times 2s$, we get
 \begin{equation*}
  \label{3.91}
  \begin{array}{ll}
  \fiint_{Q \cap \mch}  \left| \frac{\vlh(\tz) - \avgs{\vlh}{Q \cap \mch}}{r} \right|\ d\tz & \apprle \frac{1}{|Q \cap \mch|} \iint_{Q \cap \mch} |\vlh(\tz)| \ d\tz \\
  & \apprle \frac{1}{|Q \cap \mch|} \iint_{Q \cap \mch \cap \elam} |\vlh(\tz)| \ d\tz + \frac{1}{|Q \cap \mch|} \iint_{Q \cap \mch \setminus \elam} |\vlh(\tz)| \ d\tz.
  \end{array}
 \end{equation*}

 By construction of $\vlh$ in \eqref{lipschitz_extension}, we have $\vlh = v_h$ on $\elam$.  On $\ors \setminus \elam$, we can apply Lemma \ref{lemma3.7} to obtain the following bound:
 \begin{equation*}
  \begin{array}{ll}
   \fiint_{Q \cap \mch}  \left| \frac{\vlh(\tz) - \avgs{Q \cap \mch}{\vlh}}{r} \right|\ d\tz & \apprle  \frac{1}{r^n s} \iint_{\ors} |v_h(\tz)| \ d\tz +   \frac{1}{|Q \cap \mch|} \iint_{Q \cap \mch	 \setminus \elam} \rho \la\ d\tz \\
   & \apprle \lbr \frac{\ka}{s}\rbr^{\frac{n}{2}} \frac{1}{s} \| v_h \|_{L^1(\ors)} + \rho \la\\
   & \apprle o(1).
  \end{array}
 \end{equation*}
\end{description}
\end{description}
This completes the proof of the Lipschitz regularity.
\end{proof}

\subsection{Two crucial estimates}

We shall now prove the first crucial estimate which holds on each time slice. 
\begin{lemma}
 \label{pre_crucial_lemma}
 For any $i \in \NN$ and any $0 < \ve \leq 1$, there exists a positive constant $C{(n,p,q, \lamot,m_e)}$ such that for  almost every $t \in [\mft-s,\mft+s]$, there holds
 \begin{equation}
 \label{3.120}
  \left| \int_{B_{\rho}(\mfx)} (v(x,t) - v^i) \vl(x,t) \Psi_i(x,t) \ dx \right| \leq C \lbr  \frac{\la^p}{\ve} |4Q_i| + \ve |4B_i| |v^i|^2\rbr. 
 \end{equation}
 
\end{lemma}

\begin{proof}
Let us fix any $t \in [\mft-s,\mft+s]$,  $i \in \NN$  and take $\Psi_i(y,\tau) \vlh(y,\tau)$ as a test function in \eqref{eqone} and \eqref{eqtwo}. Further integrating the resulting expression over $ \left(t_i - \ka \left(\frac34 r_i\right)^2 , t\right)$ or $(\mft-s,t)$ depending on the location of $\frac34Q_i$, along with making use of  the fact that $\Psi_i(y,t_i - \ka (3r_i/4)^2) = 0$ or $\vlh(y,\mft-s) =0$, we get for  any $a\in \RR$, the equality
%
\begin{equation}
 \label{3.123}
 \begin{array}{ll}
  \int_{B_{\rho}(\mfx)} \lbr[(]  (v_h - a)  \Psi_i \vlh \rbr (y,t) \ dy & = \int_{\max \lbr[\{]t_i - \ka \left(\frac34 r_i\right)^2,\mft-s\rbr[\}]}^t \int_{B_{\rho}(\mfx)} \pa_t \left(  (v_h - a) \Psi_i \vlh \right) (y,\tau) \ dy \ d\tau \\
  & = \int_{\max \lbr[\{]t_i - \ka \left(\frac34 r_i\right)^2,\mft-s\rbr[\}]}^t \int_{B_{\rho}(\mfx)} \pa_t \left( [\varphi-\phi]_h  \Psi_i \vlh  - a \Psi_i \vlh \right) (y,\tau) \ dy \ d\tau \\
  & = \int_{\max \lbr[\{]t_i - \ka \left(\frac34 r_i\right)^2,\mft-s\rbr[\}]}^t \int_{B_{\rho}(\mfx)} \iprod{[\aa(y,\tau,\nabla \phi)]_h - [\aa(y,\tau,\nabla \varphi)]_h}{\nabla (\Psi_i \vlh)} \ dy \ d\tau  \\
  & \qquad +\ \int_{\max \lbr[\{]t_i - \ka \left(\frac34 r_i\right)^2,\mft-s\rbr[\}]}^t \int_{B_{\rho}(\mfx)} \iprod{[|\vec{f}|^{p-2} \vec{f}+|\vec{g}|^{p-2} \vec{g}]_h}{\nabla (\Psi_i \vlh)} \ dy \ d\tau  \\
  & \qquad -\  \int_{\max \lbr[\{]t_i - \ka \left(\frac34 r_i\right)^2,\mft-s\rbr[\}]}^t \int_{B_{\rho}(\mfx)} a \pa_t \lbr \Psi_i \vlh\rbr \ dy \ d\tau.
 \end{array}
\end{equation}

We can estimate $|\nabla ( \Psi_i \vl)|$ using the chain rule and \descref{W9}{W9}, to get
\begin{equation}
 \label{3.126}
 \begin{array}{ll}
 |\nabla ( \Psi_i \vlh)| 
 & \apprle \frac{1}{r_i} |\vl| + |\nabla \vl|.
 \end{array}
\end{equation}
Similarly, we can estimate $\left|\pa_t\lbr \Psi_i \vl \right)\right|$ using the chain rule and \descref{W9}{W9}, to get
\begin{align}
 \left| \pa_t \lbr \Psi_i \vl\rbr\right| & \apprle  \frac{1}{\ka r_i^2} |\vl| + |\pa_t \vl|.\label{3.128}
\end{align}

 Let us take $a=v_h^i$ in the \eqref{3.123} followed by letting $h \searrow 0$ and making use of \eqref{3.126} and \eqref{abounded},  we get
 \begin{equation}
  \label{first_1}
  \begin{array}{ll}
   \left| \int_{B_{\rho}(\mfx)} \lbr[(] (v - v^i) \Psi_i \vl \rbr (y,t) \ dy \right| & \apprle J_1 + J_2 + J_3,
  \end{array}
 \end{equation}

 where we have set 
 \begin{align}
  J_1& := \frac{1}{r_i} \iint_{\ors} \lbr |\nabla \varphi|^{p-1}+|\nabla \phi|^{p-1}+|\vec{f}|^{p-1}  + |\vec{g}|^{p-1} \rbr  |\vl|   \lsb{\chi}{\frac34Q_i\cap \ors} \ dy \ d\tau, \nonumber \\
  J_2& :=  \iint_{\ors} \lbr |\nabla \varphi|^{p-1}+|\nabla \phi|^{p-1} +|\vec{f}|^{p-1}  + |\vec{g}|^{p-1}\rbr  |\nabla \vl|   \lsb{\chi}{\frac34Q_i\cap \ors} \ dy \ d\tau,\nonumber \\
  J_3&:= \iint_{\ors} |v-v^i| | \pa_t (\Psi_i \vl)|  \lsb{\chi}{\frac34Q_i\cap \ors} \ dy \ d\tau.\label{bound_J_3_1}
 \end{align}

 Let us now estimate each of the terms as follows: 
 \begin{description}
  \item[Bound for $J_1$:] If $\rho \leq r_i$,  we can directly use H\"older's inequality, Lemma \ref{lemma3.7} and \descref{W4}{W4}, to find that for any $\ve \in (0,1]$, there holds 
  \begin{equation}
   \label{bound_I_1_rho_leq_r_i}
   \begin{array}{ll}
    J_1 & \apprle \la |Q_i| \lbr \fiint_{16Q_i} \lbr |\nabla \varphi|^{q}+|\nabla \phi|^q+|\vec{f}|^q + |\vec{g}|^q\rbr  \lsb{\chi}{\ors}  \ dy \ d\tau\rbr^{\frac{p-1}{q}} \apprle  \frac{\la^p}{\ve} |4Q_i|.
   \end{array}
  \end{equation}
  In the case $r_i \leq \rho$, we make use of \eqref{lemma3.10_bound3}, \descref{W4}{W4} along with the fact $|Q_i| = |B_i| \times 2\la^{2-p} r_i^2$, to get
  \begin{equation}
  \label{bound_I_1_rho_geq_r_i}
   \begin{array}{ll}
    J_1
    & {\apprle} \frac{1}{r_i} \lbr \frac{r_i \la}{\ve} + \frac{\ve}{\la r_i} |v^i|^2  \rbr |4Q_i|\lbr \fiint_{4Q_i} \lbr |\nabla \varphi|^{q}+|\nabla \phi|^q+|\vec{f}|^q + |\vec{g}|^q\rbr  \lsb{\chi}{\ors}\ dy \ d\tau \rbr^{\frac{p-1}{q}} \\
    & {\apprle} \frac{1}{r_i} \lbr \frac{r_i \la}{\ve} + \frac{\ve}{\la r_i}  |v^i|^2  \rbr |4Q_i|\la^{p-1} 
     {\apprle} \frac{\la^p}{\ve} |4Q_i| +  \ve |4B_i| |v^i|^2.	
   \end{array}
  \end{equation}

Thus combining \eqref{bound_I_1_rho_geq_r_i} and \eqref{bound_I_1_rho_leq_r_i}, we get
\begin{equation}
 \label{bound_I_1}
 J_1 \apprle \frac{\la^p}{\ve} |4Q_i| +  \lsb{\chi}{r_i \leq \rho}\ve |4B_i| |v^i|^2,
\end{equation}
where we have set $\lsb{\chi}{r_i \leq \rho} = 1$ if  $r_i \leq \rho$ and $\lsb{\chi}{r_i \leq \rho}=0$ else.

  \item[Bound for $J_2$:] In this case, we can directly use Lemma \ref{lemma3.9} and \descref{W4}{W4} to get for any $\ve \in (0,1]$, the bound
  \begin{equation}
   \label{bound_I_22}
    J_2 
     \apprle \frac{\la|4Q_i|}{\ve} \lbr \fiint_{4Q_i}\lbr |\nabla \varphi|^{q}+|\nabla \phi|^q+|\vec{f}|^q + |\vec{g}|^q\rbr  \lsb{\chi}{\ors} \ dy \ d\tau \rbr^{\frac{p-1}{q}}
   \apprle \frac{\la^p}{\ve} |4Q_i|.
  \end{equation}
  \item[Bound for $J_3$:] Substituting  \eqref{lemma3.10_bound4}, \eqref{lemma3.11.bound2} and \descref{W9}{W9} into \eqref{3.128},  for any $\ve \in (0,1]$, there holds
  \begin{equation}
   \label{bound_3_1}
   \begin{array}{ll}
    |\pa_t(\Psi_i \vl)(z)| 
    & \apprle \frac{1}{\ka r_i^2} \lbr \frac{r_i \la}{\ve}  + \frac{\ve}{r_i \la} |v^i|^2 \rbr + \frac{1}{\ka r_i^2} \min\{r_i,\rho\} \la \approx \frac{1}{\ka r_i^2} \lbr \frac{r_i \la}{\ve}  + \frac{\ve}{r_i \la} |v^i|^2 \rbr. 
   \end{array}
  \end{equation}
Making use of  \eqref{bound_3_1} in the expression for $J_3$ in \eqref{bound_J_3_1}, we get
  \begin{equation*}
   \label{bound_I_3}
   \begin{array}{ll}
    J_3 & \apprle \frac{1}{\ka r_i^2} \lbr \frac{r_i \la}{\ve}  + \frac{\ve}{r_i \la} |v^i|^2 \rbr\iint_{\frac34Q_i} |v-v^i|  \lsb{\chi}{\ors}  \ dy \ d\tau.
   \end{array}
  \end{equation*}
  We can now proceed similarly to \eqref{3.30} to get
\begin{equation}
 \label{bound_I33}
 \begin{array}{ll}
 J_3 & \apprle \frac{1}{\ka r_i^2} \lbr \frac{r_i \la}{\ve}  + \frac{\ve}{r_i \la} |v^i|^2 \rbr r_i \la |Q_i| 
   \apprle \frac{\la^p }{\ve} |4Q_i| + \ve |4B_i| |v^i|^2.
 \end{array}
\end{equation}

 \end{description}

 Substituting the estimates \eqref{bound_I_1}, \eqref{bound_I_22} and \eqref{bound_I33} into \eqref{first_1} gives the proof of  \eqref{3.120}. 
\end{proof}

We now come to  essentially the most important estimate which will be needed to prove the difference estimate:
\begin{lemma}
 \label{crucial_lemma}
 There exists a positive constant $C{(n,p,q, \lamot,m_e)}$ such that the following estimate holds for every $t \in [-s,s]$:
 \begin{equation}
 \label{3122}
  \int_{B_{\rho}(\mfx) \setminus \elam^t} (|v|^2 - |v - \vl|^2)(x,t) \ dx \geq - C  \la^p |\RR^{n+1} \setminus \elam|.
 \end{equation}

\end{lemma}

\begin{proof}
 Let us fix any $t\in [\mft-s,\mft+s]$ and any point $x \in B_{\rho}(\mfx) \setminus \elam^t$.  Now define
 \begin{equation*}
  \tTh := \left\{  i \in \NN: \spt(\Psi_i) \cap B_{\rho}(\mfx) \times \{t\} \neq \emptyset \quad \text{and} \quad |v| + |\vl| \neq 0 \quad  \text{on}\ \spt(\Psi_i)\cap (B_{\rho}(\mfx) \times \{t\}) \right\}.
 \end{equation*}

 Hence we only need to consider $i \in \tTh$.  
 Note that $\sum_{i \in \tTh} \Psi_i(\cdot, t) \equiv 1$ on $\RR^n \cap \elam^t$, we can rewrite the left-hand side of \eqref{3122} as 
 \begin{equation*}
  \label{I_1}
  \begin{array}{ll}
  \int_{B_{\rho}(\mfx) \setminus \elam^t} (|v|^2 - |v - \vl|^2)(x,t) \ dx & = \sum_{i \in \tTh} \int_{B_{\rho}(\mfx)} \Psi_i (|v|^2 - |v - \vl|^2) \ dx \\
  & = \sum_{i \in \tTh} \int_{B_{\rho}(\mfx)} \Psi_i(z) \lbr |v^i|^2  + 2 \vl (v - v^i) \rbr \ dx - \sum_{i \in \tTh} \int_{B_{\rho}(\mfx)} \Psi_i(z) |\vl - v^i|^2 \ dx\\
  & := J_1 + J_2.
  \end{array}
 \end{equation*}
%
 \begin{description}[leftmargin=*]
  \item[Estimate of $J_1$:] Using \eqref{3.120}, we get
  \begin{equation}
   \label{I_1_1}
   J_1 \apprge \sum_{i \in \tTh} \int_{B_{\rho}(\mfx)} \Psi_i(z)  |v^i|^2 \ dz - \ve \sum_{i \in \tTh} |4B_i| |v^i|^2 - \sum_{i \in \tTh} \frac{\la^p}{\ve} |4Q_i|.
  \end{equation}
  From \eqref{lipschitz_extension_one}, we have $v^i = 0$ whenever $\spt(\Psi_i) \cap B_{\rho}(\mfx)^c \neq \emptyset$. Hence we only have to sum over all those $i \in \tTh$ for which $\spt(\Psi_i) \subset B_{\rho}(\mfx)\times [\mft-s,\infty)$.  In this case, we make use of a suitable choice for $\ve \in (0,1]$, and use \descref{W7}{W7} along with \descref{W8}{W8}, to estimate \eqref{I_1_1} from below to get
  \begin{equation}
 \label{bound_I1}
 J_1 \apprge -\la^p |\RR^{n+1} \setminus \elam|. 
\end{equation}

  \item[Estimate of $J_2$:] For any $x \in B_{\rho}(\mfx) \setminus \elam^t$, we have from \descref{W10}{W10} that $\sum_{j} \Psi_j(x,t) = 1$, which gives
  \begin{equation}
   \label{I_2_1}
   \begin{array}{ll}
    \Psi_i(z) |\vl(z) - v^i|^2  
    & \apprle \Psi_i(z)  \sum_{j \in A_i} |\Psi_j(z)|^2  \lbr v^j - v^i\rbr^2 \\
    & \overset{\redlabel{3.104.a}{a}}{\apprle} \min\{ \rho, r_i\}^2 \la^2.
   \end{array}
  \end{equation}
To obtain \redref{3.104.a}{a} above, we made use of Lemma \ref{lemma3.8}  along with \descref{W13}{W13}.  Substituting \eqref{I_2_1} into the expression for $J_2$ and using $|Q_i| = |B_i| \times 2\ka r_i^2$, we get
\begin{equation}
 \label{bound_I_2}
J_2 
\apprle 
 \sum_{i \in \tTh_1} \left|B_i\right| \frac{\ka r_i^2}{\ka} \la^2 
 \apprle \la^p |\RR^{n+1} \setminus \elam|.
\end{equation}
 \end{description}
Substituting \eqref{bound_I1} and \eqref{bound_I_2} into \eqref{I_1}, we get
\begin{equation*}
 \label{bound_I}
 J_2 \apprge - \la^p |\RR^{n+1} \setminus \elam|.
\end{equation*}
This completes the proof of the lemma.
\end{proof}

\section{Comparison estimates}
\label{section_four}

Before we state the main difference estimates, let us define the the approximations that we will make and recall some useful results in existing literature. Let us fix the point
%
\begin{equation*}\label{dev_mfz}\mfz = (\mfx,\mft) \in \overline{\Om} \times (-T,T).\end{equation*} 


\subsection{Approximations}

Let $u$ be a weak solution of \eqref{main} and   consider the unique weak solution  $w \in C^0\lbr I_{4\rho}(\mft);L^2(\Om_{4\rho}(\mfx)\rbr  \cap L^{p}\lbr I_{4\rho}(\mft);W^{1,p}(\Om_{4\rho}(\mfx)\rbr$ solving
\begin{equation}
 \label{wapprox_int}
\left\{ \begin{array}{rcll}
  w_t - \dv \aa(x,t,\nabla w) &=& 0 & \quad \text{in} \ K_{4\rho}(\mfz),\\
  w &=&u & \quad \text{on} \ \partial_p K_{4\rho}(\mfz).
 \end{array}\right. 
\end{equation}
This is possible, since \eqref{main} shows $u \in L^{p}\lbr I_{4\rho}(\mft);W^{1,p}(\Om_{4\rho}(\mfx)\rbr$ and $\ddt{u} \in \lbr  W^{1,p}(K_{4\rho}(\mfz)) \rbr'$ in the sense of distribution.

Recalling the notation from \eqref{def_aaa},   we will need to make another approximation to \eqref{wapprox_int}:
\begin{equation}
 \label{vapprox_bnd}
\left\{ \begin{array}{rcll}
  v_t - \dv \aaa_{B_{3\rho}(\mfx)}(\nabla v,t) &=& 0 & \quad \text{in} \ K_{3\rho}(\mfz),\\
  v &=& w & \quad \text{on} \ \pa_p K_{3\rho}(\mfz),
 \end{array}\right. 
\end{equation}
which admits a unique weak solution $v \in C^0\lbr I_{3\rho}(\mft);L^2(\Om_{3\rho}(\mfx)\rbr \cap L^{p}\lbr I_{3\rho}(\mft);W^{1,p}(\Om_{3\rho}(\mfx)\rbr$ since Proposition \ref{ext_sol} is applicable.
%
%
%

\subsection{Interior Lipschitz regularity}
In the case $K_{3\rho}(\mfz) = Q_{3\rho}(\mfz)$, i.e., we are in the interior case, then we have the following interior Lipschitz regularity for \eqref{vapprox_bnd} (see \cite[Theorem 5.1 and Theorem 5.2]{DiB1}):
\begin{lemma}
\label{existence_ov}
There exists a weak solution $v \in  C^0\lbr I_{3\rho}(\mft);L^2(B_{3\rho}(\mfx)\rbr \cap L^{p}\lbr I_{2\rho}(\mft);W^{1,p}(\Om_{2\rho}^{+}(\mfx)\rbr$ solving \eqref{vapprox_bnd}. Furthermore, there holds
\begin{gather}
\sup_{Q_{\rho}(\mfz)} |\nabla v| \leq C_{(n,p,\lamot)} \lbr \fiint_{Q_{2\rho}(\mfz)} |\nabla v|^{p} \ dz \rbr^{\frac{1}{p}}.
\end{gather}
\end{lemma}

\subsection{Boundary Lipschitz regularity}
\label{bnd_app}
In the boundary case, we may not have Lipschitz regularity for solutions of \eqref{vapprox_bnd} up to the boundary in general. In order to overcome this difficulty, we need to make one further approximation in which we consider a  weak solution $\ov \in  C^0\lbr I_{2\rho}(\mft);L^2(\Om_{2\rho}^{+}(\mfx)\rbr \cap L^{p}\lbr I_{2\rho}(\mft);W^{1,p}(\Om_{2\rho}^{+}(\mfx)\rbr$ solving
\begin{equation*}
 \label{Vapprox_bnd}
\left\{ \begin{array}{rcll}
  \ov_t - \dv \aaa_{{B_{3\rho}(\mfx)}}(\nabla \ov,t) &=& 0 & \quad \text{in} \ Q_{2\rho}^{+}(\mfz),\\
  \ov &=& 0 & \quad \text{on} \  T_{2\rho}(\mfz).
 \end{array}\right. 
\end{equation*}

From \cite[Theorem 1.6]{lieberman1993boundary}, the following important lemma holds:
\begin{lemma}
\label{existence_oV}
There exists a weak solution $\ov \in  C^0\lbr I_{2\rho}(\mft);L^2(\Om_{2\rho}^{+}(\mfx)\rbr \cap L^{p}\lbr I_{2\rho}(\mft);W^{1,p}(\Om_{2\rho}^{+}(\mfx)\rbr$ solving \eqref{Vapprox_bnd}. Furthermore, there holds
\begin{gather*}
\sup_{Q^{+}_{\rho}(\mfz)} |\nabla \ov| \leq C(n,p,\lamot) \lbr \fiint_{Q^{+}_{2\rho}(\mfz)} |\nabla \ov|^{p} \ dz \rbr^{\frac{1}{p}}.
\end{gather*}
\end{lemma}

\subsection{First comparison estimate}
In this subsection, we will prove a improved difference estimate between solutions of \eqref{main} and \eqref{wapprox_int}. 
\begin{theorem}
\label{first_difference_estimate}
Let $\de >0$ be given and Let $u$ be a weak solution of \eqref{main} and $w$ be the unique  weak solution of \eqref{wapprox_int}, then there exists an $\be_1 = \be_1(\lamot,p,n,m_e,\de) \in (0,1)$ such that for any $\be \in (0,\be_1)$, the following estimate holds:
\begin{equation*}
\label{lipschitz_difference_estimate}
\fiint_{K_{4\rho}(\mfz)} |\nabla u - \nabla w|^{p-\be} \ dz \leq \de \fiint_{K_{4\rho}(\mfz)} |\nabla u|^{p-\be} \ dz + C(n,p,\be,\lamot,\de) \fiint_{K_{4\rho}(\mfz)} |\bff|^{p-\be}.
\end{equation*}
\end{theorem}
\begin{proof}
Consider the following cut-off function $\zv \in C^{\infty}(\mft -(4\rho)^2,\infty)$ such that $0 \leq \zv(t) \leq 1$ and 
\begin{equation*}
\label{def_zv}
\zv(t) = \left\{ \begin{array}{ll}
                1 & \text{for} \ t \in (\mft -(4\rho)^2+\ve,\mft +(4\rho)^2-\ve),\\
                0 & \text{for} \ t \in (-\infty,\mft -(4\rho)^2)\cup (\mft +(4\rho)^2,\infty).
                \end{array}\right.
\end{equation*}	

It is easy to see that 
\begin{equation*}
\label{bound_zv}
\begin{array}{c}
\zv'(t) = 0 \ \txt{for} \  t \in (-\infty,\mft -(4\rho)^2) \cup (\mft -(4\rho)^2+\ve,\mft +(4\rho)^2-\ve)\cup (\mft +(4\rho)^2,\infty), \\
|\zv'(t)| \leq \frac{c}{\ve}\  \txt{for} \  t \in (\mft -(4\rho)^2,\mft -(4\rho)^2+\ve) \cup (\mft +(4\rho)^2-\ve,\mft -(4\rho)^2).
\end{array}
\end{equation*}
Without loss of generality, we shall always take $2h \leq \ve$, since we will take limits in the following order $\lim_{\ve \rightarrow 0} \lim_{h \rightarrow 0}$.

Let us apply the results of Section \ref{section_three} with $\varphi = u$, $\phi = w$, $\vec{f} = \bff$ and $\vec{g} = 0$ over $Q_{\rho,s}(\mfz) = K_{4\rho}(\mfz)$ to get a Lipschitz test function $\vlh$ satisfying Lemma \ref{crucial_lemma}. From Lemma \ref{lemma3.15}, we have $\vlh \in C^{0,1}(K_{4\rho}(\mfz))$ and thus we shall use $\vlh(z) \zv(t)$ as a test function to get
\begin{equation*}
\label{5.13}
\begin{array}{c}
L_1 + L_2 :=\iint_{K_{4\rho}(\mfz)} \ddt{[u-w]_h} \vlh \zv \ dx \ dt + \iint_{K_{4\rho}(\mfz)} \iprod{[\aa(x,t,\nabla u) - \aa(x,t,\nabla w)]_h}{\nabla \vlh} \zv \ dx \ dt \\
\hspace*{6cm} = \iint_{K_{4\rho}(\mfz)} \iprod{[|\bff|^{p-2} \bff]_h}{\nabla \vlh} \zv  \ dx \ dt =:L_3.
\end{array}
\end{equation*}

Let us recall from Section \ref{section_three} the following: for a fixed $1<q < p-2\be$, we have
\[
g(z) := \mm \lbr \lbr[[] |\nabla u - \nabla w|^q  + |\nabla u|^q + |\nabla w|^q + |\bff|^q \rbr[]] \lsb{\chi}{K_{4\rho}(\mfz)}\rbr^{\frac{1}{q}},
\]
where $\mm$ is as defined in \eqref{par_max} and $\elam = \{z \in \RR^{n+1}: g(z) \leq \la\}$. \emph{Note that at this point, we have not really made any choice of $\be$.}

From the strong Maximal function estimates (see \cite[Lemma 7.9]{Gary} for the proof), we have
\begin{equation}
\label{max_g_fde}
\begin{array}{rcl}
\| g\|_{L^{p-\be}(\RR^{n+1})} & \apprle &   \|\nabla u \|_{L^{p-\be}(K_{4\rho}(\mfz))}+ \|\nabla w\|_{L^{p-\be}(K_{4\rho}(\mfz))} + \| \bff\|_{L^{p-\be}(K_{4\rho}(\mfz))} + \| \nabla u - \nabla w\|_{L^{p-\be}(K_{4\rho}(\mfz))}  \\
& \apprle &    \|\nabla u\|_{L^{p-\be}(K_{4\rho}(\mfz))} + \| \bff\|_{L^{p-\be}(K_{4\rho}(\mfz))} + \| \nabla u - \nabla w\|_{L^{p-\be}(K_{4\rho}(\mfz))}.
\end{array}
\end{equation}

\begin{description}
\item[Estimate for $L_1$:] \begin{equation}
   \label{6.38}
   \begin{array}{ll}
    L_1     & = \int_{\mft - (4\rho)^2}^{\mft + (4\rho)^2} \int_{{\Om_{4\rho}(\mfx)}} \dds{v_h(y,s)}  \vlh(y,s)\zv(s) \ dy\ ds \\
    & \qquad  + \int_{\mft - (4\rho)^2}^{\mft + (4\rho)^2} \int_{{\Om_{4\rho}(\mfx)}}  \frac{d{\lbr \lbr[[](v_h)^2 - (\vlh - v_h)^2\rbr[]]\zv(s) \rbr }}{ds} \ dy \ ds \\
    & \qquad - \int_0^{\mft + (4\rho)^2} \int_{{\Om_{4\rho}(\mfx)}} \dds{\zv} \lbr v_h^2 - (\vlh - v_h)^2 \rbr \ dy \ ds\\
    & := J_2 + J_1(\mft + (4\rho)^2) - J_1(\mft - (4\rho)^2) - J_3,
   \end{array}
  \end{equation}
  where we have set 
  \begin{equation*}
  \label{def_i_1}J_1(s) := \frac12 \int_{{\Om_{4\rho}(\mfx)}} ( (v_h)^2 - (\vlh - v_h)^2 ) (y,s) \zv(s) \ dy.
  \end{equation*}
Note that $J_1(\mft - (4\rho)^2) = J_1(\mft + (4\rho)^2) =0$ since $\zv(\mft - (4\rho)^2) =\zv(\mft + (4\rho)^2) =  0$.

Form Lemma \ref{lemma3.14} applied with $\vt =1$, we have the bound
\begin{equation}
    \label{6.39}
    \begin{array}{ll}
     |J_2| & \apprle \iint_{K_{4\rho}(\mfz)\setminus \elam}   \left| \dds{\vlh}  (\vlh-v_h)\right| \ dy \ ds \apprle\la^p |\RR^{n+1} \setminus \elam| .
    \end{array}
   \end{equation}

\item[Estimate for $L_2$:] We split $L_2$ and make use of the fact that $\vlh(z) = v_h(z)$ for all $z\in \elam\cap K_{4\rho}(\mfz).$ 
\begin{equation}
\label{5.17}
\begin{array}{ll}
L_2 & = \iint_{K_{4\rho}(\mfz)\cap \elam} \iprod{[\aa(x,t,\nabla u) - \aa(x,t,\nabla w)]_h}{\nabla \vlh} \zv\ dz  \\
& \qquad + \iint_{K_{4\rho}(\mfz)\setminus \elam} \iprod{[\aa(x,t,\nabla u) - \aa(x,t,\nabla w)]_h}{\nabla \vlh} \zv\ dz\\
& = \iint_{K_{4\rho}(\mfz)\cap \elam} \iprod{[\aa(x,t,\nabla u) - \aa(x,t,\nabla w)]_h}{\nabla [u-w]_h} \zv\ dz  \\
& \qquad + \iint_{K_{4\rho}(\mfz)\setminus \elam} \iprod{[\aa(x,t,\nabla u) - \aa(x,t,\nabla w)]_h}{\nabla \vlh} \zv\ dz\\
& = L_2^1 + L_2^2.
\end{array}
\end{equation}

\begin{description}
\item[Estimate for $L_2^1$:] Using ellipticity, we get
\begin{equation}
\label{5.18}
\begin{array}{ll}
L_2^1 & = \iint_{K_{4\rho}(\mfz)\cap \elam} \iprod{[\aa(x,t,\nabla u) - \aa(x,t,\nabla w)]_h}{\nabla [u-w]_h} \zv\ dz \\
& \apprge \iint_{K_{4\rho}(\mfz) \cap \elam} \left[|\nabla u-w|^2 \lbr |\nabla u|^2 + |\nabla w|^2 \rbr\right]_h^{\frac{p-2}{2}} \zv \ dz.
\end{array}
\end{equation}

\item[Estimate for $L_2^2$:] Using the bound  from Lemma \ref{lemma3.9}, we get
\begin{equation}
\label{5.19}
\begin{array}{ll}
L_2^2 & \apprle \iint_{K_{4\rho}(\mfz)\setminus \elam} \lbr[|][\aa(x,t,\nabla u) - \aa(x,t,\nabla w)]_h\rbr[|] |\nabla \vlh| \ dz\\
& \apprle \la \iint_{K_{4\rho}(\mfz)\setminus \elam} |\nabla [u]_h|^{p-1} + |\nabla [w]_h|^{p-1}\ dz.
\end{array}
\end{equation}

\end{description}

\item[Estimate for $L_3$:] Analogous to estimate for $L_2$, we split $L_3$ into integrals over $\elam$ and $\elam^c$ followed by making use of \eqref{lipschitz_extension} and Lemma \ref{lemma3.9} to get
\begin{equation}
\label{5.20}
L_3 \apprle \iint_{K_{4\rho}(\mfz) \cap \elam} [|\bff|^{p-1}]_h |\nabla [u-w]_h| \ dz + \la \iint_{K_{4\rho}(\mfz)\setminus \elam}  |[\bff]_h|^{p-1}\ dz.
\end{equation}

%
%
%

\end{description}

Combining  \eqref{6.39} into \eqref{6.38} followed by \eqref{5.18} and \eqref{5.19} into \eqref{5.17} and making use of \eqref{6.38}, \eqref{6.39} and \eqref{5.20}, we get 
\begin{equation}
\label{combined_1}
\begin{array}{l}
- \int_{\mft - (4\rho)^2}^{\mft +(4\rho)^2} \int_{{\Om_{4\rho}(\mfx)}} \dds{\zv} \lbr v_h^2 - (\vlh - v_h)^2 \rbr \ dy \ ds +  \iint_{K_{4\rho}(\mfz) \cap \elam} |\nabla [u-w]_h|^2 \lbr |\nabla [u]_h|^2 +|\nabla [w]_h|^2 \rbr^{\frac{p-2}{2}}\zv \ dz \\
\hfill \apprle \iint_{K_{4\rho}(\mfz)\cap \elam} [|\bff|^{p-1}]_h{|\nabla [u-w]_h|} \ dz +  \la \iint_{K_{4\rho}(\mfz)\setminus \elam} |\nabla [u]_h|^{p-1} + |\nabla [w]_h|^{p-1}+ |[\bff]_h|^{p-1}\ dz \\
\qquad\qquad\qquad\qquad + \la^p |\RR^{n+1} \setminus \elam| .
\end{array}
\end{equation}

In order to estimate $- \int_{\mft - (4\rho)^2}^{\mft +(4\rho)^2} \int_{{\Om_{4\rho}(\mfx)}} \dds{\zv} \lbr v_h^2 - (\vlh - v_h)^2 \rbr \ dy \ ds$, we
take limits first in $h\searrow 0$ followed by $\ve \searrow 0$ to get
\begin{equation}
\label{5.22}
\begin{array}{ll}
- \int_{\mft - (4\rho)^2}^{\mft +(4\rho)^2} \int_{{\Om_{4\rho}(\mfx)}} \dds{\zv} \lbr v_h^2 - (\vlh - v_h)^2 \rbr \ dy \ ds  \xrightarrow{\lim\limits_{\ve \searrow 0}\lim\limits_{h \searrow 0}} & \int_{{\Om_{4\rho}(\mfx)}} (v^2 - (\vl - v)^2 )(x,\mft + (4\rho)^2) \ dx \\
& - \int_{{\Om_{4\rho}(\mfx)}} (v^2 - (\vl - v)^2 )(x,\mft - (4\rho)^2) \ dx.
\end{array}
\end{equation}

For the second term on the right of \eqref{5.22}, we observe that on $\elam$, we have $\vl = v$ and on $\elam^c$ and we also have $\vl(\cdot,\mft - (4\rho)^2) = v(\cdot,\mft - (4\rho)^2) = 0$. Thus, the second term vanishes because on $\elam$, we can use the initial boundary condition and on $\elam^c$, it is zero by construction. Thus we get
\begin{equation}
\label{5.23}
- \int_{\mft - (4\rho)^2}^{\mft +(4\rho)^2} \int_{{\Om_{4\rho}(\mfx)}} \dds{\zv} \lbr v_h^2 - (\vlh - v_h)^2 \rbr \ dy \ ds  \xrightarrow{\lim\limits_{\ve \searrow 0} \lim\limits_{h \searrow 0}}  \int_{{\Om_{4\rho}(\mfx)}} (v^2 - (\vl - v)^2 )(x,\mft + (4\rho)^2) \ dx.
\end{equation}
Thus using \eqref{5.23} into \eqref{combined_1} gives
\begin{equation}
\label{combined_1_1}
\begin{array}{l}
 \int_{{\Om_{4\rho}(\mfx)}} (v^2 - (\vl - v)^2 )(x,\mft + (4\rho)^2) \ dx +  \iint_{K_{4\rho}(\mfz) \cap \elam} |\nabla [u-w]_h|^2 \lbr |\nabla [u]_h|^2 +|\nabla [w]_h|^2 \rbr^{\frac{p-2}{2}}\zv \ dz \\
\hfill \apprle \iint_{K_{4\rho}(\mfz)\cap \elam} [|\bff|^{p-1}]_h{|\nabla [u-w]_h|} \ dz +  \la \iint_{K_{4\rho}(\mfz)\setminus \elam} |\nabla [u]_h|^{p-1} + |\nabla [w]_h|^{p-1}+ |[\bff]_h|^{p-1}\ dz\\
\qquad\qquad\qquad\qquad + \la^p |\RR^{n+1} \setminus \elam| .
\end{array}
\end{equation}

In fact, if we consider a cut-off function $\zv^{t_0} (\cdot)$ for some $t_0 \in (\mft - (4\rho)^2,\mft + (4\rho)^2)$, where 
\begin{equation*}
\zv^{t_0}(t) = \left\{ \begin{array}{ll}
                1 & \text{for} \ t \in (-t_0+\ve,t_0-\ve),\\
                0 & \text{for} \ t \in (-\infty,-t_0)\cup (t_0,\infty).
                \end{array}\right.
\end{equation*}
we get the following analogue of \eqref{combined_1_1}
\begin{equation}
\label{combined_1_new_11}
\begin{array}{l}
\int_{{\Om_{4\rho}(\mfx)}} (v^2 - (\vl - v)^2 )(x,t_0) \ dx +  \int_{-t_0}^{t_0} \int_{{\Om_{4\rho}(\mfx)} \cap \elam^t} |\nabla (u-w)|^2 \lbr |\nabla u|^2 +|\nabla w|^2 \rbr^{\frac{p-2}{2}} \ dz  \\
\qquad \qquad  \apprle  \iint_{K_{4\rho}(\mfz)\cap \elam} |\bff|^{p-1}|\nabla (u-w)| \ dz +  \la \iint_{K_{4\rho}(\mfz)\setminus \elam} |\nabla u|^{p-1} + |\nabla w|^{p-1}+ |\bff|^{p-1}\ dz\\
\qquad\qquad\qquad\qquad + \la^p |\RR^{n+1} \setminus \elam| .
\end{array}
\end{equation}


Using  Lemma \ref{crucial_lemma},  we get for any $t \in (\mft - (4\rho)^2,\mft + (4\rho)^2)$, the estimate
\begin{equation}
 \label{4.13}
  \begin{array}{ll}
    \int_{{\Om_{4\rho}(\mfx)}} | (v)^2 - (\vl - v)^2 | (y,t) \ dy   
   & \apprge \int_{\elam^t} | v (x,t)|^2 \ dx   - \la^p |\RR^{n+1} \setminus \elam|.
  \end{array}
 \end{equation}

 Since $\int_{\elam^t} | v (x,t)|^2 \ dx$ occurs on the left hand side and is positive, we can ignore this term. Thus combining \eqref{4.13} with \eqref{combined_1_new_11}, we get
\begin{equation}
\label{fully_combined}
\begin{array}{ll}
\iint_{K_{4\rho}(\mfz) \cap \elam} |\nabla (u-w)|^2 \lbr |\nabla u|^2 +|\nabla w|^2 \rbr^{\frac{p-2}{2}} \ dx \ dt &\apprle \iint_{K_{4\rho}(\mfz)\cap \elam} |\bff|^{p-1}|\nabla (u-w)| \ dz \\
 &\quad  +  \la \iint_{K_{4\rho}(\mfz)\setminus \elam} |\nabla u|^{p-1} + |\nabla w|^{p-1}+ |\bff|^{p-1}\ dz\\
& \quad + \la^p |\RR^{n+1} \setminus \elam| .
\end{array}
\end{equation} 
 
Let us now multiply \eqref{fully_combined}  with $\la^{-1-\be}$ and integrating over $(0,\infty)$ with respect to $\la$, we get
 \begin{equation}
 \label{K_expression}
K_1 +K_2\apprle K_3 + K_4,
 \end{equation}
 where we have set
 \begin{equation*}
  \begin{array}{@{}r@{}c@{}l@{}}
  K_1 \ &:=& \ \int_{0}^{\infty} \la^{-1-\be} \iint_{K_{4\rho}(\mfz) \cap \elam} |\nabla (u-w)|^2 \lbr |\nabla u|^2 +|\nabla u|^2 \rbr^{\frac{p-2}{2}}\ dz \ d\la,\\
  K_2 \ &:=& \ \int_{0}^{\infty} \la^{-1-\be} \iint_{K_{4\rho}(\mfz) \cap \elam} |\bff|^{p-1}|\nabla (u-w)| \ dz \ d\la,  \\
  K_3 \ &:=& \ \int_0^{\infty} \la^{-\be} \iint_{K_{4\rho}(\mfz)\setminus \elam} |\nabla u|^{p-1} + |\nabla w|^{p-1}+ |\bff|^{p-1}\ dz \ d\la, \\
  K_4\  &:=& \ \int_{0}^{\infty} \la^{-1-\be}  \la^p |\RR^{n+1} \setminus \elam| \  d\la. \\
  \end{array}
 \end{equation*}

 \begin{description}
 \item[Estimate for $K_1$:] Applying Fubini, we get
 \begin{equation*}
 \label{3.13}
 K_1 \apprge \frac{1}{\be} \iint_{K_{4\rho}(\mfz)} g(z)^{-\be}|\nabla (u-w)|^2 \lbr |\nabla u|^2 +|\nabla u|^2 \rbr^{\frac{p-2}{2}}\ dz.
 \end{equation*}

 Using Young's inequality along with \eqref{abounded} and  \eqref{max_g_fde}, we get for any $\ep_1 >0$, the estimate
 \begin{equation}
 \label{bound_K_1}
 \begin{array}{rcl}
 \iint_{K_{4\rho}(\mfz)} |\nabla u - \nabla w|^{p-\be} \ dz & \apprle &  C(\ep_1) \be K_1  + \ep_1 \iint_{K_{4\rho}(\mfz)}  |\nabla u-\nabla w|^{p-\be} + |\nabla u|^{p-\be} \ dz  \\
 & & \qquad +  C_3(\ep_1) \iint_{K_{4\rho}(\mfz)} |\bff|^{p-\be}\, dz.
 \end{array}
 \end{equation}

 \item[Estimate for $K_2$:] Again by Fubini, we get
 \begin{equation*}
 \label{3.18}
 K_2  = \frac{1}{\be} \iint_{K_{4\rho}(\mfz)} g(z)^{-\be} \iprod{|\bff|^{p-2} \bff}{\nabla u - \nabla w} \ dz.
 \end{equation*}
From the definition of $g(z)$, we see that for $z \in K_{4\rho}(\mfz)$, we have $g(z) \geq |\nabla u - \nabla w|(z)$, which implies $g(z)^{-\be} \leq |\nabla u - \nabla w|^{-\be}(z)$. Now we apply Young's inequality, for any $\ep_2 >0$, we get
\begin{equation}
\label{3.19}
\begin{array}{ll}
K_2 & \leq \frac{1}{\be} \iint_{K_{4\rho}(\mfz)} |\nabla u - \nabla w|^{1-\be} |\bff|^{p-1} \ dz \\
& \apprle \frac{C(\ep_2)}{\be} \iint_{K_{4\rho}(\mfz)} |\bff|^{p-\be} \ dz + \frac{\ep_2}{\be} \iint_{K_{4\rho}(\mfz)} |\nabla u - \nabla w|^{p-\be} \ dz.
\end{array}
\end{equation}

 \item[Estimate for $K_3$:] Again applying Fubini, we get
 \begin{equation*}
 \label{3.20}
 K_3 = \frac{1}{p-\be} \iint_{K_{4\rho}(\mfz)} g(z)^{1-\be} \lbr |\nabla u|^{p-1} +|\nabla w|^{p-1} + |\bff|^{p-1} \rbr \ dz.
 \end{equation*}
 Applying Young's inequality followed by making use of \eqref{max_g_fde}, we get
 \begin{equation}
 \label{3.21}
 \begin{array}{ll}
 K_3  
 & \apprle \iint_{K_{4\rho}(\mfz)}  |\nabla u - \nabla w|^{p-\be} + |\nabla u|^{p-\be} + |\bff|^{p-\be}\ dz. \\
 \end{array}
 \end{equation}

 \item[Estimate for $K_4$:] Applying the layer cake representation followed by using \eqref{max_g_fde}, we get
 \begin{equation}
 \label{3.22}
 \begin{array}{ll}
 K_4 & = \frac{1}{p-\be} \iint_{\RR^{n+1}} g(z)^{p-\be} \ dz \\
 & \apprle \iint_{K_{4\rho}(\mfz)} |\nabla u - \nabla w|^{p-\be} + |\nabla u|^{p-\be}  + |\bff|^{p-\be} \ dz.
 \end{array}
 \end{equation}

 \end{description}

 We now combine \eqref{bound_K_1}, \eqref{3.19}, \eqref{3.21} and \eqref{3.22} into \eqref{K_expression}, we get
 \begin{equation*}
 \label{3.23}
 \begin{array}{rcl}
  \iint_{K_{4\rho}(\mfz)} |\nabla u - \nabla w|^{p-\be} \ dz & \apprle &  \lbr[[] \ep_1 + C(\ep_1) (\ep_2 + \be) \rbr[]] \iint_{K_{4\rho}(\mfz)}  |\nabla u-\nabla w|^{p-\be}  \ dz   +  C(\ep_1,\ep_2,\be) \iint_{K_{4\rho}(\mfz)} |\bff|^{p-\be}\, dz \\
 && \qquad +  \lbr [[]\ep_1 + C(\ep_1) \be\rbr[]] \iint_{K_{4\rho}(\mfz)} |\nabla u|^{p-\be} \ dz.\\
 \end{array}
 \end{equation*}
 Choosing $\ep_1$ small followed by $\ep_2$ and $\be$, for any $\de > 0$, we get a $\be_1 = \be_1(n,p,\lamot,\de)$ such that for any $\be \in (0,\be_1)$, there holds
 \[
 \fiint_{K_{4\rho}(\mfz)} |\nabla u - \nabla w|^{p-\be} \ dz \leq \de \fiint_{K_{4\rho}(\mfz)} |\nabla u|^{p-\be} \ dz + C(n,p,\be,\lamot,\de) \fiint_{K_{4\rho}(\mfz)} |\bff|^{p-\be}.
 \]
This completes the proof of the theorem.
 \end{proof}

\subsection{Second comparison estimate}
In this subsection, we will prove an improved comparison estimate between solutions of \eqref{wapprox_int} and \eqref{vapprox_bnd}. 
\begin{theorem}
\label{second_difference_estimate}
Let $w$ be a weak solution of \eqref{wapprox_int} and $v$ be the unique  weak solution of \eqref{vapprox_bnd}, then there exists an $\be_2 = \be_2(\lamot,p,n,m_e,\ve) \in (0,1)$ such that for any $\be \in (0,\be_2)$ and $\ve >0$, there exists a  $C = C(n,\lamot,p,\ve)>0$ and $\sigma_1 = \sigma_1(\lamot,p)$ such that the following estimate holds:
\begin{equation*}
\label{lipschitz_difference_estimate_B}
\lbr \fiint_{K_{3\rho}(\mfz)} |\nabla w - \nabla v|^{p-\be} \ dz\rbr^{\frac{p}{p-\be}} \leq \ve \lbr \fiint_{K_{4\rho}(\mfz)} |\nabla w|^{p-\be} \ dz \rbr^{1+\be \tilde{\vt}_1}  + C [\aa]_{2,R_0}^{\sigma_1} \lbr \fiint_{K_{4\rho}(\mfz)} |\nabla w|^{p-\be} \ dz \rbr^{\frac{(1+\be\tilde{\vt}_1)(1+\be\tilde{\vt}_2) p}{p+\be}}.
\end{equation*}
Here $\tilde{\vt}_1$ and $\tilde{\vt}_2$ are from Lemma \ref{high_weak} and Lemma \ref{high_very_weak}.
\end{theorem}

\begin{proof}

Since we are in the setting of weak solutions, from \cite[Lemma 2.8]{duong2017global}, we have the following estimate: for any $\ve >0$, there exists a $C = C(n,\lamot,p,\ve)>0$ and $\sigma_1 = \sigma_1(\lamot,p)$ such that
\begin{equation}
\label{est1B}
\fiint_{K_{3\rho}(\mfz)} |\nabla w - \nabla v|^{p} \ dz \leq \ve \fiint_{K_{3\rho}(\mfz)} |\nabla w|^{p} \ dz + C [\aa]_{2,R_0}^{\sigma_1} \lbr \fiint_{K_{3\rho}(\mfz)} |\nabla w|^{p+\be} \ dz \rbr^{\frac{p}{p+\be}}.
\end{equation}

Since $w$ solves the homogeneous equation \eqref{wapprox_int}, we can control the right hand side of \eqref{est1B} by using Lemma \ref{high_weak} and  Lemma \ref{high_very_weak}. For $\be_2 := \min\{ \tilde{\be}_1, \tilde{\be}_2\}$, for any $\be \in (0,\be_2)$, there holds
\begin{equation}
\label{est2B}
\fiint_{K_{3\rho}(\mfz)} |\nabla w|^{p+\be} \ dz \apprle \lbr \fiint_{K_{\frac72\rho}(\mfz)} |\nabla w|^{p} \ dz \rbr^{1+\be \tilde{\vt}_1} \apprle \lbr \fiint_{K_{4\rho}(\mfz)} |\nabla w|^{p-\be} \ dz \rbr^{\lbr 1+\be \tilde{\vt}_2\rbr \lbr  1+\be \tilde{\vt}_1\rbr},
\end{equation}
where $\tilde{\vt}_1$ and $\tilde{\vt}_2$ are from Lemma \ref{high_weak} and Lemma \ref{high_very_weak} respectively.

We now combine \eqref{est2B} and \eqref{est1B} to get
\begin{equation*}
\label{est3B}
\fiint_{K_{3\rho}(\mfz)} |\nabla w - \nabla v|^{p} \ dz \leq \ve \lbr \fiint_{K_{4\rho}(\mfz)} |\nabla w|^{p-\be} \ dz \rbr^{1+\be \tilde{\vt}_1}  + C [\aa]_{2,R_0}^{\sigma_1} \lbr \fiint_{K_{4\rho}(\mfz)} |\nabla w|^{p-\be} \ dz \rbr^{\frac{(1+\be\tilde{\vt}_1)(1+\be\tilde{\vt}_2) p}{p+\be}}.
\end{equation*}

A simple application of H\"older's inequality now gives
\begin{equation*}
\lbr \fiint_{K_{3\rho}(\mfz)} |\nabla w - \nabla v|^{p-\be} \ dz\rbr^{\frac{p}{p-\be}} \leq \ve \lbr \fiint_{K_{4\rho}(\mfz)} |\nabla w|^{p-\be} \ dz \rbr^{1+\be \tilde{\vt}_1}  + C [\aa]_{2,R_0}^{\sigma_1} \lbr \fiint_{K_{4\rho}(\mfz)} |\nabla w|^{p-\be} \ dz \rbr^{\frac{(1+\be\tilde{\vt}_1)(1+\be\tilde{\vt}_2) p}{p+\be}}.
\end{equation*}

\end{proof}

\subsection{Interior approximation estimate}
In this subsection, we will prove the interior approximation lemma:

\begin{lemma}
Let $\be \in (0,\be_0)$ for $\be_0= \min\{\be_1,\be_2\}$ where $\be_1$ is from Theorem \ref{first_difference_estimate} and $\be_2$ is from Theorem \ref{second_difference_estimate}.  For each $\ve >0$, there exists a $\de>0$ (possibly depending on $\ve$) such that the following holds true: Assume $u$ is a weak solution of \eqref{main} satisfying 
\begin{equation}\label{hyp_1_int}
\fiint_{K_{4\rho}(\mfz)} |\nabla u|^{p-\be} \ dz \leq 1,
\end{equation}
then under the condition 
\begin{equation}\label{cond_1_int}
\fiint_{K_{4\rho}(\mfz)} |\bff|^{p-\be} \ dz \leq \de^{p-\be},
\end{equation}
there exists a weak solution $v$ to \eqref{vapprox_bnd} satisfying
\begin{equation}
\label{conc_1_int}
\|\nabla v\|_{L^{\infty}(K_{2\rho}(\mfz))} \apprle 1, \txt{and} \fiint_{K_{2\rho}(\mfz)} |\nabla u - \nabla v|^{p-\be} \ dz \leq \ve^{p-\be}.
\end{equation}
\end{lemma}

\begin{proof}
Let us prove each of the assertions of \eqref{conc_1_int} as follows: 
\begin{description}[leftmargin=*]
\item[First estimate in \eqref{conc_1_int}:] From Lemma \ref{existence_ov}, we have existence of a weak solution $v$ solving \eqref{vapprox_bnd} satisfying the estimate
\[
\|\nabla v\|_{L^{\infty}(K_{2\rho}(\mfz))}^p \leq C(n,p,\lamot)  \fiint_{Q_{3\rho}(\mfz)} |\nabla v|^{p} \ dz.
\]
We now estimate the right hand side as follows:
\begin{equation*}
\label{5.47}
\begin{array}{rcl}
\fiint_{Q_{3\rho}(\mfz)} |\nabla v|^{p} \ dz & \overset{\redlabel{5.47.a}{a}}{\apprle} & \fiint_{Q_{3\rho}(\mfz)} |\nabla v -\nabla w|^{p} \ dz + \fiint_{Q_{3\rho}(\mfz)} |\nabla w|^{p} \ dz\\
& \overset{\redlabel{5.47.b}{b}}{\apprle} & ( 1 + \ve) \lbr \fiint_{K_{4\rho}(\mfz)} |\nabla w|^{p-\be} \ dz \rbr^{1+\be \tilde{\vt}_1}  + C [\aa]_{2,S_0}^{\sigma_1} \lbr \fiint_{K_{4\rho}(\mfz)} |\nabla w|^{p-\be} \ dz \rbr^{\frac{(1+\be\tilde{\vt}_1)(1+\be\tilde{\vt}_2) p}{p+\be}} \\
& \overset{\redlabel{5.47.c}{c}}{\apprle} & ( 1 + \ve) \lbr \fiint_{K_{4\rho}(\mfz)} |\nabla w - \nabla u|^{p-\be} \ dz \rbr^{1+\be \tilde{\vt}_1}  \\
&& + C [\aa]_{2,S_0}^{\sigma_1} \lbr \fiint_{K_{4\rho}(\mfz)} |\nabla w - \nabla u|^{p-\be} \ dz \rbr^{\frac{(1+\be\tilde{\vt}_1)(1+\be\tilde{\vt}_2) p}{p+\be}} \\
&& + ( 1 + \ve) \lbr \fiint_{K_{4\rho}(\mfz)} |\nabla u|^{p-\be} \ dz \rbr^{1+\be \tilde{\vt}_1}  + C [\aa]_{2,S_0}^{\sigma_1} \lbr \fiint_{K_{4\rho}(\mfz)} |\nabla u|^{p-\be} \ dz \rbr^{\frac{(1+\be\tilde{\vt}_1)(1+\be\tilde{\vt}_2) p}{p+\be}}.
\end{array}
\end{equation*}
In order to obtain \redref{5.47.a}{a}, we made use of triangle inequality, to obtain \redref{5.47.b}{b}, we made use of Theorem \ref{second_difference_estimate} along with Lemma \ref{high_very_weak} and finally to obtain \redref{5.47.c}{c}, we applied triangle inequality. 

We can control $\fiint_{K_{4\rho}(\mfz)} |\nabla w - \nabla u|^{p-\be} \ dz$ using Theorem \ref{first_difference_estimate} along with making use of \eqref{hyp_1_int} and \eqref{cond_1_int} and observing that $[\aa]_{2,S_0} \leq C(p,\lamot)$, we get 
%
\begin{equation*}
\fiint_{Q_{3\rho}(\mfz)} |\nabla v|^{p} \ dz \leq C(\lamot,n,p,\de).
\end{equation*}

This proves the first assertion of \eqref{conc_1_int}. 
\item[Second estimate in \eqref{conc_1_int}:]
Using triangle inequality, we get
\[
\fiint_{K_{2\rho}(\mfz)} |\nabla u - \nabla v|^{p-\be} \ dz  \apprle \fiint_{K_{2\rho}(\mfz)} |\nabla u - \nabla w|^{p-\be} \ dz  + \fiint_{K_{2\rho}(\mfz)} |\nabla w - \nabla v|^{p-\be} \ dz.
\]
Each of the above terms can be controlled using Theorem \ref{first_difference_estimate} and Theorem \ref{second_difference_estimate} along with \eqref{hyp_1_int} followed by choosing $\de$ sufficiently small (depending on $\ve$) and  $\ga$ sufficiently small such that $(\aa,\Om)$ is $(\ga,S_0)$-vanishing to get the desired conclusion. 
\end{description}

\end{proof}

\subsection{Boundary approximation estimate}
In this subsection, we will prove the boundary approximation lemma:
\begin{lemma}
\label{lem_bnd_app}
Let $\be \in (0,\be_0)$ be fixed and let $w$ be a weak solution of \eqref{wapprox_int} satisfying 
\[
\fiint_{K_{3\rho}(\mfz)} |\nabla w|^{p-\be} \ dz \leq 1,
\]
then for any $\ve >0$, there exists a small $\ga = \ga(\lamot,n,p,\ve)>0$ such that if $(\aa,\Om)$ is $(\ga,S_0)$ vanishing, then there exists a weak solution $\ov$ of \eqref{Vapprox_bnd} whose zero extension to $Q_{2\rho}(\mfz)$ satisfies
\begin{equation}\label{lem5.9.a}
\fiint_{Q_{2\rho}^+(\mfz)} |\nabla \ov|^{p} \ dz \leq 1, \txt{and} \fiint_{K_{\rho}(\mfz)} |\nabla w - \nabla \ov|^{p} \ dz \leq \ve^{p}.
\end{equation}
\end{lemma}
\begin{proof}
From Lemma \ref{high_very_weak}, we see that the 
\[
\fiint_{K_{3\rho}(\mfz)} |\nabla w|^{p-\be} \ dz \leq 1\quad  \Rightarrow\quad \fiint_{K_{3\rho}(\mfz)} |\nabla w|^{p} \ dz \apprle 1.
\]
Hence we can apply \cite[Lemma 3.8]{BOS1} to get an $\ga = \ga(\lamot,n,p,\ve)>0$ such that if $(\aa,\Om)$ is $(\ga,S_0)$ vanishing, then there exists a weak solution $\ov$ of \eqref{Vapprox_bnd} whose zero extension to $Q_{2\rho}(\mfz)$ satisfies
\[
\fiint_{Q_{2\rho}^+(\mfz)} |\nabla \ov|^{p} \ dz \leq 1, \txt{and} \fiint_{K_{\rho}(\mfz)} |\nabla w - \nabla \ov|^{p} \ dz \leq \ve^{p}.
\]
This completes the proof of the lemma.
\end{proof}

\begin{corollary}
Let $\be \in (0,\be_0)$ be fixed and let $w$ be a weak solution of \eqref{wapprox_int} satisfying 
\[
\fiint_{K_{3\rho}(\mfz)} |\nabla w|^{p-\be} \ dz \leq 1,
\]
then for any $\ve >0$, there exists a small $\ga = \ga(\lamot,n,p,\ve)>0$ such that if $(\aa,\Om)$ is $(\ga,S_0)$ vanishing, then there exists a weak solution $\ov$ of \eqref{Vapprox_bnd} whose zero extension to $Q_{2\rho}(\mfz)$ satisfies
%
\[
 \|\nabla \ov\|_{L^{\infty}(Q_{\rho}(\mfz))} \leq C(n,p,\lamot), \txt{and} \fiint_{K_{\rho}(\mfz)} |\nabla w - \nabla \ov|^{p-\be} \ dz \leq \ve^{p-\be}.
\]
\end{corollary}
\begin{proof}
All the hypothesis of Lemma \ref{lem_bnd_app} is satisfied. Thus the first conclusion follows directly by combining \eqref{lem5.9.a} along with Lemma \ref{existence_ov} and the second conclusion follows by a simple application of H\"older's inequality to \eqref{lem5.9.a}.
\end{proof}

\section{Proof of Theorem \ref{thm6.1}.}
\label{section_five}

Consider the following cut-off function $\zv \in C^{\infty}(-T,\infty)$ such that $0 \leq \zv(t) \leq 1$ and 
\begin{equation*}
\label{def_zv.3}
\zv(t) = \left\{ \begin{array}{ll}
                1 & \text{for} \ t \in (-T+\ve,T-\ve),\\
                0 & \text{for} \ t \in (-\infty,-T)\cup (T,\infty).
                \end{array}\right.
\end{equation*}	

It is easy to see that 
\begin{equation*}
\label{bound_zv.3}
\begin{array}{c}
\zv'(t) = 0 \ \txt{for} \  t \in (-\infty,-T) \cup (-T+\ve,T-\ve)\cup (T,\infty), \\
|\zv'(t)| \leq \frac{c}{\ve}\  \txt{for} \  t \in (-T,-T+\ve) \cup (T-\ve,-T).
\end{array}
\end{equation*}
Without loss of generality, we shall always take $2h \leq \ve$, since we will take limits in the following order $\lim_{\ve \rightarrow 0} \lim_{h \rightarrow 0}$. 

Since $u =0$ on $\pa \Om \times (-T,T)$, we can apply the results of Section \ref{section_three} with $\varphi = u$, $\phi = 0$, $\vec{f} = \bff$ and $\vec{g} = 0$ over $Q_{\rho,s} = \Om \times (-T,T)$ to get a Lipschitz test function $\vlh$ satisfying Lemma \ref{crucial_lemma}. Thus we shall use $\vlh(z) \zv(t)$ as a test function in \eqref{main} to get
\begin{equation*}
\label{5.13.3}
\begin{array}{c}
\iint_{\Om_T} \ddt{[u]_h} \vlh \zv \ dx \ dt + \iint_{\Om_T} \iprod{[\aa(x,t,\nabla u)]_h}{\nabla \vlh} \zv \ dx \ dt  = \iint_{\Om_T} \iprod{[|\bff|^{p-2} \bff]_h}{\nabla \vlh} \zv  \ dx \ dt,
\end{array}
\end{equation*}
which we write as $L_1 + L_2 = L_3.$

Let us recall from Section \ref{section_three} the following: for a fixed $1<q < p-2\be$, we have
\[
g(z) := \mm \lbr \lbr[[] |\nabla u|^q  +  |\bff|^q \rbr[]] \lsb{\chi}{\Om_T}\rbr^{\frac{1}{q}},
\]
where $\mm$ is as defined in \eqref{par_max} and $\elam = \{z \in \RR^{n+1}: g(z) \leq \la\}$. \emph{Note that at this point in the proof, we have not really made any choice of $\be$.}

From the strong Maximal function estimates (see \cite[Lemma 7.9]{Gary} for the proof), we have
\begin{equation}
\label{max_g_fde.3}
\begin{array}{rcl}
\| g\|_{L^{p-\be}(\RR^{n+1})} & \leq &  C(n) \lbr \||\nabla u| \lsb{\chi}{\Om_T} \|_{L^{p-\be}(\RR^{n+1})}+ \| |\bff|\lsb{\chi}{\Om_T}\|_{L^{p-\be}(\RR^{n+1})}  \rbr.
\end{array}
\end{equation}

\begin{description}
\item[Estimate for $L_1$:] \begin{equation}
   \label{6.38.3}
   \begin{array}{ll}
    L_1     & = \int_{-T}^{T} \int_{\Om} \dds{u_h(y,s)}  \vlh(y,s)\zv(s) \ dy\ ds \\
    & = \int_{-T}^{T} \int_{\Om\setminus \elam^{s}}  \dds{\vlh}  (\vlh-u_h) \zv(s)\ dy \ ds +    \int_{-T}^{T} \int_{\Om}  \frac{d{\lbr \lbr[[](u_h)^2 - (\vlh - u_h)^2\rbr[]]\zv(s) \rbr }}{ds} \ dy \ ds \\
    & \qquad - \int_0^{T} \int_{\Om} \dds{\zv} \lbr u_h^2 - (\vlh - u_h)^2 \rbr \ dy \ ds\\
    & := J_2 + J_1(T) - J_1(-T) - J_3,
   \end{array}
  \end{equation}
  where we have set 
  \begin{equation*}
  \label{def_i_1.3}J_1(s) := \frac12 \int_{\Om} ( (u_h)^2 - (\vlh - u_h)^2 ) (y,s) \zv(s) \ dy.
  \end{equation*}
Note that $J_1(-T) = J_1(T) =0$ since $\zv(-T) =\zv(T) =  0$.

Form Lemma \ref{lemma3.14} applied with $\vt =1$, we have the bound
\begin{equation}
    \label{6.39.3}
    \begin{array}{ll}
     |J_2| & \apprle \iint_{\Om_T\setminus \elam}   \left| \dds{\vlh}  (\vlh-u_h)\right| \ dy \ ds \apprle\la^p |\RR^{n+1} \setminus \elam| .
    \end{array}
   \end{equation}

\item[Estimate for $L_2$:] We split $L_2$ and make use of the fact that $\vlh(z) = u_h(z)$ for all $z\in \elam\cap \Om_T.$ 
\begin{equation}
\label{5.17.3}
\begin{array}{ll}
L_2 & = \iint_{\Om_T\cap \elam} \iprod{[\aa(x,t,\nabla u)]_h}{\nabla \vlh} \zv\ dz + \iint_{\Om_T\setminus \elam} \iprod{[\aa(x,t,\nabla u)]_h}{\nabla \vlh} \zv\ dz\\
& = \iint_{\Om_T\cap \elam} \iprod{[\aa(x,t,\nabla u)]_h}{\nabla [u]_h} \zv\ dz + \iint_{\Om_T\setminus \elam} \iprod{[\aa(x,t,\nabla u)]_h}{\nabla \vlh} \zv\ dz\\
& = L_2^1 + L_2^2.
\end{array}
\end{equation}

\begin{description}
\item[Estimate for $L_2^1$:] Using ellipticity from \eqref{abounded}, we get
\begin{equation}
\label{5.18.3}
\begin{array}{ll}
L_2^1 
& \apprge \iint_{\Om_T \cap \elam} \left[|\nabla u|^p \right]_h \zv \ dz.
\end{array}
\end{equation}

\item[Estimate for $L_2^2$:] Using the bound  from Lemma \ref{lemma3.9}, we get
\begin{equation}
\label{5.19.3}
\begin{array}{ll}
L_2^2 
& \apprle \la \iint_{\Om_T\setminus \elam} |\nabla [u]_h|^{p-1}\ dz.
\end{array}
\end{equation}

\end{description}

\item[Estimate for $L_3$:] Analogous to estimate for $L_2$, we split $L_3$ into integrals over $\elam$ and $\elam^c$ followed by making use of \eqref{lipschitz_extension} and Lemma \ref{lemma3.9} to get
\begin{equation}
\label{5.20.3}
L_3 \apprle \iint_{\Om_T \cap \elam} [|\bff|^{p-1}]_h |\nabla [u]_h| \ dz + \la \iint_{\Om_T\setminus \elam}  |[\bff]_h|^{p-1}\ dz.
\end{equation}

%
%
%

\end{description}

Combining  \eqref{6.39.3} into \eqref{6.38.3} followed by \eqref{5.18.3} and \eqref{5.19.3} into \eqref{5.17.3} and making use of \eqref{6.38.3}, \eqref{6.39.3} and \eqref{5.20.3}, we get 
\begin{equation}
\label{combined_1.3}
\begin{array}{rcl}
- \int_{-T}^{T} \int_{\Om} \dds{\zv} \lbr u_h^2 - (\vlh - u_h)^2 \rbr \ dy \ ds +  \iint_{\Om_T \cap \elam} |\nabla [u]_h|^p \zv \ dz &\apprle& \iint_{\Om_T\cap \elam} [|\bff|^{p-1}]_h{|\nabla [u]_h|} \ dz   \\
& & + \la \iint_{\Om_T\setminus \elam}  |[\bff]_h|^{p-1}\ dz \\
& &  + \la^p |\RR^{n+1} \setminus \elam| .
\end{array}
\end{equation}

In order to estimate $- \int_{-T}^{T} \int_{\Om} \dds{\zv} \lbr u_h^2 - (\vlh - u_h)^2 \rbr \ dy \ ds$, we
take limits first in $h\searrow 0$ followed by $\ve \searrow 0$ to get
\begin{equation}
\label{5.22.3}
\begin{array}{ll}
- \int_{-T}^{T} \int_{\Om} \dds{\zv} \lbr u_h^2 - (\vlh - u_h)^2 \rbr \ dy \ ds  \xrightarrow{\lim\limits_{\ve \searrow 0}\lim\limits_{h \searrow 0}} & \int_{\Om} (u^2 - (\vl - u)^2 )(x,T) \ dx \\
& - \int_{\Om} (u^2 - (\vl - u)^2 )(x,-T) \ dx.
\end{array}
\end{equation}

For the second term on the right of \eqref{5.22.3}, we observe that on $\elam$, we have $\vl = u$ and on $\elam^c$ and we also have $\vl(\cdot,-T) = u(\cdot,-T) = 0$ using the initial condition. Thus, the second term on the right of \eqref{5.22.3} vanishes from which we get
\begin{equation*}
\label{5.23.3}
- \int_{-T}^{T} \int_{\Om} \dds{\zv} \lbr u_h^2 - (\vlh - u_h)^2 \rbr \ dy \ ds  \xrightarrow{\lim\limits_{\ve \searrow 0} \lim\limits_{h \searrow 0}}  \int_{\Om} (u^2 - (\vl - u)^2 )(x,T) \ dx.
\end{equation*}
Thus using \eqref{5.23} into \eqref{combined_1.3} gives
\begin{equation}
\label{combined_1_1.3}
\begin{array}{rcl}
 \int_{\Om} (u^2 - (\vl - u)^2 )(x,T) \ dx +  \iint_{\Om_T \cap \elam} |\nabla u|^p  \ dz & \apprle & \iint_{\Om_T\cap \elam} |\bff|^{p-1}|\nabla u| \ dz + \la \iint_{\Om_T\setminus \elam} |\bff|^{p-1}\ dz\\
&&\qquad + \la^p |\RR^{n+1} \setminus \elam|.
\end{array}
\end{equation}

In fact, if we consider a cut-off function $\zv^{t_0} (\cdot)$ for some $t_0 \in (-T,T)$, where 
\begin{equation*}
\zv^{t_0}(t) = \left\{ \begin{array}{ll}
                1 & \text{for} \ t \in (-t_0+\ve,t_0-\ve),\\
                0 & \text{for} \ t \in (-\infty,-t_0)\cup (t_0,\infty).
                \end{array}\right.
\end{equation*}
we get the following analogue of \eqref{combined_1_1.3}
\begin{equation}
\label{combined_1_new_11.3}
\begin{array}{rcl}
\int_{\Om} (v^2 - (\vl - v)^2 )(x,t_0) \ dx +  \int_{-t_0}^{t_0} \int_{\Om \cap \elam^t} |\nabla u)|^p \ dz & \apprle & \iint_{\Om_T\cap \elam} |\bff|^{p-1}|\nabla u| \ dz + \la \iint_{\Om_T\setminus \elam} |\bff|^{p-1}\ dz\\
&&\qquad + \la^p |\RR^{n+1} \setminus \elam| .
\end{array}
\end{equation}


Using  Lemma \ref{crucial_lemma},  we get for any $t \in (-T,T)$, the estimate
\begin{equation}
 \label{4.13.3}
  \begin{array}{ll}
    \int_{\Om} | (u)^2 - (\vl - u)^2 | (y,t) \ dy   
   & \apprge \int_{\elam^t} | u (x,t)|^2 \ dx   - \la^p |\RR^{n+1} \setminus \elam|.
  \end{array}
 \end{equation}

 Since $\int_{\elam^t} | u (x,t)|^2 \ dx$ occurs on the left hand side and is positive, we can ignore this term. Thus combining \eqref{4.13.3} with \eqref{combined_1_new_11.3}, we get
\begin{equation}
\label{fully_combined.3}
\begin{array}{ll}
\iint_{\Om_T \cap \elam} |\nabla u|^p \ dz &\apprle \iint_{\Om_T\cap \elam} |\bff|^{p-1}|\nabla u| \ dz + \la \iint_{\Om_T\setminus \elam}  |\bff|^{p-1}\ dz + \la^p |\RR^{n+1} \setminus \elam| .
\end{array}
\end{equation} 
 
Let us now multiply \eqref{fully_combined.3}  with $\la^{-1-\be}$ and integrating over $(0,\infty)$ with respect to $\la$, we get
 \begin{equation}
 \label{K_expression.3}
K_1 \apprle K_2+ K_3 + K_4,
 \end{equation}
 where we have set
 \begin{equation*}
  \begin{array}{@{}r@{}c@{}l@{}}
  K_1 \ &:=& \ \int_{0}^{\infty} \la^{-1-\be} \iint_{\Om_T \cap \elam} |\nabla (u-w)|^2 \lbr |\nabla u|^2 +|\nabla u|^2 \rbr^{\frac{p-2}{2}}\ dz \ d\la,\\
  K_2 \ &:=& \ \int_{0}^{\infty} \la^{-1-\be} \iint_{\Om_T \cap \elam} |\bff|^{p-1}|\nabla (u-w)| \ dz \ d\la,  \\
  K_3 \ &:=& \ \int_0^{\infty} \la^{-\be} \iint_{\Om_T\setminus \elam} |\nabla u|^{p-1} + |\nabla w|^{p-1}+ |\bff|^{p-1}\ dz \ d\la, \\
  K_4\  &:=& \ \int_{0}^{\infty} \la^{-1-\be}  \la^p |\RR^{n+1} \setminus \elam| \  d\la. \\
  \end{array}
 \end{equation*}

 \begin{description}
 \item[Estimate for $K_1$:] Applying Fubini, we get
 \begin{equation*}
 \label{3.13.3}
 K_1 \apprge \frac{1}{\be} \iint_{\Om_T} g(z)^{-\be}|\nabla u|^p \ dz.
 \end{equation*}

 Using Young's inequality along with   \eqref{max_g_fde.3}, we get for any $\ep_1 >0$, the estimate
 \begin{equation}
 \label{bound_K_1.3}
 \iint_{\Om_T} |\nabla u|^{p-\be} \ dz  \apprle   C(\ep_1) \be K_1  + \ep_1 \iint_{\Om_T} |\nabla u|^{p-\be} \ dz  +  C_3(\ep_1) \iint_{\Om_T} |\bff|^{p-\be}\, dz.
 \end{equation}

 \item[Estimate for $K_2$:] Again by Fubini, we get
 \begin{equation*}
 \label{3.18.3}
 K_2  = \frac{1}{\be} \iint_{\Om_T} g(z)^{-\be} |\bff|^{p-1} \abs{\nabla u} \ dz.
 \end{equation*}
From the definition of $g(z)$, we see that for $z \in \Om_T$, we have $g(z) \geq |\nabla u |(z)$ which implies $g(z)^{-\be} \leq |\nabla u|^{-\be}(z)$. Now we apply Young's inequality, for any $\ep_2 >0$, we get
\begin{equation}
\label{3.19.3}
K_2 
 \apprle \frac{C(\ep_2)}{\be} \iint_{\Om_T} |\bff|^{p-\be} \ dz + \frac{\ep_2}{\be} \iint_{\Om_T} |\nabla u|^{p-\be} \ dz.
\end{equation}

 \item[Estimate for $K_3$:] Again applying Fubini, we get
 \begin{equation}
 \label{3.21.3}
 K_3 = \frac{1}{p-\be} \iint_{\Om_T} g(z)^{1-\be}|\bff|^{p-1} \ dz \overset{\redlabel{3.20.3.a}{a}}{\apprle} \iint_{\Om_T} |\nabla u|^{p-\be} + |\bff|^{p-\be}\ dz.
 \end{equation}
 To obtain \redref{3.20.3.a}{a}, we made use of  Young's inequality followed by \eqref{max_g_fde.3}.

 \item[Estimate for $K_4$:] Applying the layer cake representation followed by using \eqref{max_g_fde.3}, we get
 \begin{equation}
 \label{3.22.3}
 K_4  = \frac{1}{p-\be} \iint_{\RR^{n+1}} g(z)^{p-\be} \ dz  \apprle \iint_{\Om_T} |\nabla u|^{p-\be}  + |\bff|^{p-\be} \ dz.
 \end{equation}

 \end{description}

 We now combine \eqref{bound_K_1.3}, \eqref{3.19.3}, \eqref{3.21.3} and \eqref{3.22.3} into \eqref{K_expression.3}, we get
 \begin{equation*}
 \label{3.23.3}
 \begin{array}{rcl}
  \iint_{\Om_T} |\nabla u|^{p-\be} \ dz & \apprle &  \lbr[[] \ep_1 + C(\ep_1) (\ep_2 + \be) \rbr[]] \iint_{\Om_T}  |\nabla u|^{p-\be}  \ dz   +  C(\ep_1,\ep_2,\be) \iint_{\Om_T} |\bff|^{p-\be}\, dz \\
 && \qquad +  \lbr [[]\ep_1 + C(\ep_1) \be\rbr[]] \iint_{\Om_T} |\nabla u|^{p-\be} \ dz.\\
 \end{array}
 \end{equation*}
 Choosing $\ep_1$ small followed by $\ep_2$ and $\be$, we get a $\be_3 = \be_3(n,p,\lamot,\ve)$ such that for any $\be \in (0,\be_3)$, there holds
 \[
 \fiint_{\Om_T} |\nabla u|^{p-\be} \ dz \apprle_{(n,p,\be,\lamot)} \fiint_{\Om_T} |\bff|^{p-\be}\ dz.
 \]
This completes the proof of the theorem.

\section{Covering arguments}
\label{section_six}

Once we have the estimates in Section \ref{section_four} and Section \ref{section_five}, the covering arguments can be proved in the standard way. We will only provide a brief sketch of the estimates. 

\begin{remark}\label{remark7.1}In this section and following section, let $\be_0$ be that such that for all $2\be \in (0,\be_0]$, the results in Section \ref{section_four} and Section \ref{section_five} are applicable. We will now fix an $\be$ with $2\be \leq \be_0$.\end{remark}

Let us define 
\begin{equation}
\label{al_0}
\al_0^{\frac{p-\be}{d}} := \fiint_{\Om_T} \lbr[[]|\nabla u|^{p-\be} + \lbr\frac{|\bff|}{\ga}\rbr^{p-\be} \rbr[]]\ dx \ dt.
\end{equation}
where $d$ is defined to be
\begin{equation}
\label{def_d}
d:=\left\{
\begin{array}{ll}
\frac{p-\be}{2-\be} & \text{if} \ p \geq 2, \\
 \frac{2(p-\be)}{2p-2\be+np-2n} & \text{if} \ \frac{2n}{n+2} < p < 2.
\end{array}\right.
\end{equation}

Furthermore, define 
$c_e := \lbr[[] \lbr \frac{16}{7}\rbr^n \frac{|\Om_T|}{|B_1|S_0^{n+2}} \rbr[]]^{\frac{d}{p-\be}},$
and let 
\begin{equation}\label{lalarge}
\la \geq c_e \al_0.
\end{equation}

We will also need to consider the following superlevel set:
\begin{equation*}\label{elam}\elam:= \{z \in \Om_T : |\nabla u(z)| > \la\}.\end{equation*}

The first lemma that we need is the following:
\begin{lemma}
Let $\ga \in (0,1)$ be any constant, then for any $\la$ satisfying \eqref{lalarge}, there exists a family of disjoint cylinders $\{K_{r_i}^{\la}(z_i)\}_{i \in \NN}$ with $z_i \in \elam$
and $r_i \in (0,S_0)$ such that
\begin{equation*}\label{bounds_many}\begin{array}{c}
\fiint_{K_{r_i}^{\la}(z_i)} \lbr[[] |\nabla u|^{p-\be} + \lbr\frac{|\bff|}{\ga}\rbr^{p-\be}  \rbr[]]\ dx \ dt = \la^{p-\be}, \\
\fiint_{K_{r_i}^{\la}(z_i)} \lbr[[] |\nabla u|^{p-\be} + \lbr\frac{|\bff|}{\ga}\rbr^{p-\be}  \rbr[]]\ dx \ dt < \la^{p-\be} \qquad \text{for every} \ r> r_i, \\
\elam \subset \bigcup_{i\in\NN} K_{5r_i}^{\la}(z_i).
\end{array}\end{equation*}
\end{lemma}

The proof following using standard techniques from Measure theory and Lemma \ref{measure_density} (see \cite[Pages 4311-4313]{BOS1} for the details).

Using Lemma \ref{weight_lemma} the following lemma follows:
\begin{lemma}
\label{lemma7.3}
Let $w \in A_s$ for any $s \geq \frac{p-\be}{p-2\be}$, then it is automatically in $A_1$. Then there exists a constant $c^{\ast} = c^{\ast}([w]_1,n,p)>0)$ such that there holds
\[
 w(K_{r_i}^{\la}(z_i)) \leq \frac{C}{\la^{p-\be}}   \lbr[[] \iint\limits_{K_{r_i}^{\la}(z_i) \cap \{|\nabla u| > \frac{\la}{4c^{\ast}}\}} |\nabla u|^{p-\be} w(z) \ dz + \iint\limits_{K_{r_i}^{\la}(z_i) \cap \{|\bff u| > \frac{\ga \la}{4c^{\ast}}\}} \abs{ \frac{\bff}{\ga} }^{p-\be} w(z) \ dz\rbr[]].
\]
\end{lemma}
The proof of the above Lemma is standard and we refer to \cite[Page 4114 - (3.8)]{byun2017weighted} for the necessary details.

Now making use of the a priori estimates in Section \ref{section_four} and Section \ref{section_five} and combining with the techniques of \cite[Lemma 4.3]{BOS1}, the following Lemma holds:
\begin{lemma}
\label{lemma7.4}
There exists a constant $N = N(\lamot,n,p)>1$ such that for any $\ve \in (0,1)$, there exists a small $\ga = \ga(\lamot,\ve,n,p)$ such that if $(\aa,\Om)$ is $(\ga,S_0)$ vanishing for such a small $\ga$ and some fixed $S_0>0$, then there holds
\[
\frac{\abs{\{z \in K_{5r_i}^{\la}(z_i): |\nabla u(z)| > 2N\la\}}}{|K_{r_i}^{\la}(z_i)} \leq c_{(\lamot,n,p)} \ve^{p-\be}.
\]
\end{lemma}

We can now combine Lemma \ref{lemma7.3} and Lemma \ref{lemma7.4} to prove the following weighted estimate on the level sets (again the proof follows exactly as in \cite[STEP 4 on Page 4115]{byun2017weighted} and will be omitted). 
\begin{lemma}
\label{lemma7.5}
Let $\be$ be as in Remark \ref{remark7.1}, then there holds
\[
w(E(2N\la)) \apprle \frac{\ve^{(p-2\be)\tau_1}}{\la^{p-\be}} \lbr[[] \iint\limits_{\Om_T \cap \{|\nabla u| > \frac{\la}{4c^{\ast}}\}} |\nabla u|^{p-\be} w(z) \ dz + \iint\limits_{\Om_T \cap \{|\bff u| > \frac{\ga \la}{4c^{\ast}}\}} \abs{ \frac{\bff}{\ga} }^{p-\be} w(z) \ dz\rbr[]].
\]
Here $\tau_1$ is defined as in Definition \ref{a_infinity}.
\end{lemma}

\section{Proof of Theorem \ref{main_theorem}.}
\label{section_seven}
From Lemma \ref{weight_level_set}, we get
\begin{equation}
\label{8.1}
\begin{array}{rcl}
\iint_{\Om_T} |\nabla u|^{q} w(z) \ dz 
& = & q \int_0^{c_e\al_0} (2N\la)^{q-1} w(\{ z\in \Om_T: |\nabla u|> 2N\la\}) d (2N\la)  \\
&& \qquad + q \int_{c_e\al_0}^{\infty} (2N\la)^{q-1} w(\{ z\in \Om_T: |\nabla u|> 2N\la\}) d (2N\la) \\
& =: & II_1 + II_2.
\end{array}
\end{equation}

\begin{description}[leftmargin=*]
\item[Estimate for $II_1$:]  With $d$ as defined in \eqref{def_d}, we get
\begin{equation}
\label{8.2}
\begin{array}{rcl}
II_1 & \apprle & w(\Om_T) \al_0^q \overset{\eqref{al_0}}{=} w(\Om_T) \lbr \fiint_{\Om_T} \lbr[[]|\nabla u|^{p-\be} + \lbr\frac{|\bff|}{\ga}\rbr^{p-\be} \rbr[]]\ dx \ dt \rbr^{\frac{qd}{p-\be}} \\
& \overset{\text{Theorem \ref{thm6.1}}}{\apprle} &  w(\Om_T) \lbr \fiint_{\Om_T} |\bff|^{p-\be}\ dx \ dt \rbr^{\frac{qd}{p-\be}} \\
& \overset{\text{Lemma \ref{weight_lemma}}}{\apprle} &  w(\Om_T) \lbr \frac{1}{w(\Om_T)} \iint_{\Om_T} |\bff|^{q}w(x,t)\ dx \ dt \rbr^{d} \\
& \apprle & \lbr  \iint_{\Om_T} |\bff|^{q}w(x,t)\ dx \ dt \rbr^{d}.
\end{array}
\end{equation}

\item[Estimate for $II_2$:] Since we have $q \geq p > p-\be > p-2\be$, we proceed as follows:
\begin{equation}
\label{8.3}
\begin{array}{rcl}
II_2 & \overset{\text{Lemma \ref{lemma7.5}}}{\apprle} & \ve^{(p-2\be)\tau_1}\int_1^{\infty} \frac{\la^{q-1}}{\la^{p-\be}} \iint\limits_{\Om_T \cap \{|\nabla u| > \frac{\la}{4c^{\ast}}\}} |\nabla u|^{p-\be} w(z) \ dz\ d\la  \\
&&\qquad  + \frac{\ve^{(p-2\be)\tau_1}}{\ga^{(p-2\be)\tau_1}} \int_1^{\infty} \frac{\la^{q-1}}{\la^{p-\be}} \iint\limits_{\Om_T \cap \{|\bff| > \frac{\ga\la}{4c^{\ast}}\}} |\bff|^{p-\be} w(z) \ dz \ d\la \\
& \overset{\text{Lemma \ref{weight_level_set}}}{\apprle} & \ve^{p-2\be} \iint_{\Om_T} |\nabla u|^q w(z) \ dz + C(\ga,\ve)  \iint_{\Om_T} |\bff|^q w(z) \ dz.
\end{array}
\end{equation}
\end{description}
Combining \eqref{8.2} and \eqref{8.3} into \eqref{8.1} following by choosing $\ve$ sufficiently small, the proof follows.


\end{document}